\documentclass[reqno,10pt]{amsart}

\usepackage{pinlabel}
\usepackage{graphicx}
\usepackage{epstopdf}
\usepackage{amsmath}%
\usepackage{amsfonts}%
\usepackage{amssymb}%
\usepackage{xy,amsthm,enumerate,xypic,array}

\title{Cyclic $A_\infty$ Structures and Deligne's Conjecture}
\author{Benjamin C. Ward}
\address{Purdue University, Department of Mathematics  \newline%
\indent 150 N. University St.  West Lafayette, IN 47907-2067}%
\email{wardbc@math.purdue.edu}%

\newcommand{\ds}{\displaystyle}
\newcommand{\N}{\mathbb{N}}

\newcommand{\C}{\mathbb{C}}

\newcommand{\Z}{\mathbb{Z}}

\newcommand{\tensor}{\otimes}

\newcommand{\op}{\mathcal}
\newcommand{\fr}{\mathfrak}

\newcommand{\cdc}{,\dots,}
\newcommand{\tdt}{\tensor\dots\tensor}

\input{xy}
\xyoption{all}

\numberwithin{equation}{section}

\newtheorem{theorem}{Theorem}[section]
\theoremstyle{plain}

\newtheorem*{thma}{Theorem A}
\newtheorem*{thmb}{Theorem B}
\newtheorem*{thmc}{Theorem C}
\newtheorem*{thmd}{Theorem D}

\newtheorem{corollary}[theorem]{Corollary}
\newtheorem{lemma}[theorem]{Lemma}
\newtheorem{proposition}[theorem]{Proposition}

\theoremstyle{definition}
\newtheorem{definition}[theorem]{Definition}
\newtheorem{example}[theorem]{Example}

\newtheorem{remark}[theorem]{Remark}
\newtheorem{notation}[theorem]{Notation}

\newtheorem*{Acknowledgments}{Acknowledgments}

\allowdisplaybreaks[2]

\addtocounter{MaxMatrixCols}{2}

\begin{document}

\begin{abstract}  First we describe a class of homotopy Frobenius algebras via cyclic operads which we call cyclic $A_\infty$ algebras.  We then define a suitable new combinatorial operad which acts on the Hochschild cochains of such an algebra in a manner which encodes the homotopy BV structure.  Moreover we show that this operad is equivalent to the cellular chains of a certain topological (quasi)-operad of CW complexes whose constituent spaces form a homotopy associative version of the Cacti operad of Voronov.  These cellular chains thus constitute a chain model for the framed little disks operad, proving a cyclic $A_\infty$ version of Deligne's conjecture.  This chain model contains the minimal operad of Kontsevich and Soibelman as a suboperad and restriction of the action to this suboperad recovers the results of \cite{KS} and \cite{KSch} in the unframed case.  Additionally this proof recovers the work of Kaufmann in the case of a strict Frobenius algebra.  We then extend our results to the context of cyclic $A_\infty$ categories, with an eye toward the homotopy BV structure present on the Hochschild cochains of the Fukaya category of a suitable symplectic manifold. 
\end{abstract}

\maketitle

\section*{Introduction}  Throughout we let $k$ be a field of characteristic $\neq 2$.  Given $M$, a closed oriented manifold, there are several meaningful constructions which associate a BV $k$-algebra to $M$ including:  the (shifted) homology of the free loop space of $M$ \cite{CS}, the Hochschild cohomology of the singular cochains of $M$ \cite{FMT}, and the symplectic cohomology \cite{CFH} of the unit disk cotangent bundle of $M$ \cite{Seidel2}.  Much work has been devoted to the study of the BV structures listed above and the relationships between them \cite{CJ},\cite{Seidel3},\cite{Vit1},\cite{AS1},\cite{AS2}.  However, noticing that these BV structures all arise on the level of homology/cohomology suggests that they are, as written in \cite{DCV}, merely the `shadow of a higher structure: that of a homotopy BV algebra'.  

The purpose of this paper is to investigate an emerging class of homotopy BV structures which arises in a related context.  Given a compact symplectic manifold $N$, for example the unit disk cotangent bundle of $M$, we can consider the Fukaya category $\op{F}(N)$ (see e.g. \cite{Seidel}).  Recall $\op{F}(N)$ is an $A_\infty$ category whose objects are (classes of) Lagrangian submanifolds of $N$ and which has the additional structure of an inner product
\begin{equation*}
Hom(L_1,L_2)\tensor Hom(L_2,L_1)\to k
\end{equation*}   
which is expected to be cyclically invariant.  We will call such a structure a cyclic $A_\infty$ category (see Definition $\ref{cyccatdef}$).  The cyclic invariance of the form would imply that the Hochschild cohomology of $\op{F}(N)$ is a BV algebra and would endow the Hochschild cochains of $\op{F}(N)$ with a homotopy BV structure.  This is the homotopy BV structure that will be considered here in.  As such, it would be expected that our principal objects of study would be cyclic $A_\infty$ categories and their Hochschild cohomology.  However, as we will show, the study of cyclic $A_\infty$ categories and their Hochschild cohomology can be largely performed in the context of cyclic $A_\infty$ algebras and their Hochschild cohomology.  As a result, for the sake of simplicity we conduct the bulk of our study in terms of cyclic $A_\infty$ algebras, and conclude with the categorical generalization in Section $\ref{section:cat}$.

Cyclic $A_\infty$ algebras are a particular class of homotopy Frobenius algebras, namely those which relax the associativity to an $A_\infty$ algebra structure but do not resolve the bilinear form.  Such algebras first appeared in \cite{Kont}.  Our approach will be to realize Frobenius algebras and more generally cyclic $A_\infty$ algebras as cyclic unital algebras over the cyclic unital operads $\op{A}s$ and $\op{A}_\infty$.  In particular we will take care to make the cyclic structure of the operad $\op{A}_\infty$ geometrically and combinatorially explicit.

\subsection*{Summary of Results}
Our main result is a proof of a version of Deligne's conjecture for cyclic $A_\infty$ algebras.  Since such an algebra is Frobenius up to homotopy, the Hochschild cohomology of such an algebra is a BV algebra \cite{Menichi},\cite{tradler} and hence an algebra over the homology of the framed little disks operad, $fD_2$.  We can then ask if there is a suitable chain model for $fD_2$ which admits a lift of this action, and the answer is yes (Theorem $\ref{mainthm}$).

\begin{thma}  There is a dg operad $\op{TS}_\infty$ which is a cell model for the framed little disks operad and which acts on the Hochschild cochains of a cyclic $A_\infty$ algebra in a manner compatible with the standard operations on homology/cohomology.
\end{thma}

To prove this theorem we start with the chain model of Kontsevich and Soibelman used in \cite{KS} in a proof of the (noncyclic) $A_\infty$ version as well as its underlying topological structure exhibited in \cite{KSch}, and add a topological $S^1$ action inducing the desired cyclic structure on the chain level.  This technique is inspired by Kaufmann's proof in the associative case \cite{KCyclic}.  The result is a chain model for $fD_2$ which can be realized as the cellular chains of a topological (quasi)-operad and which can be viewed as an $A_\infty$ analog of the Cacti operad of Voronov \cite{Vor}.  In particular we prove the following (Theorem $\ref{mainthm2}$).

\begin{thmb}  The cell model $\op{TS}_\infty$ is isomorphic to the cellular chains of a topological (quasi)-operad $\op{X}=\{X_n\}$ of CW complexes.  Furthermore there is a surjective homotopy equivalence from $\op{X}$ to normalized Cacti:
\begin{equation*}
\op{X}\stackrel{\sim}\longrightarrow \op{C}acti^1
\end{equation*}
\end{thmb}
The spaces $X_n$ are constructed using the cyclohedra of Bott and Taubes \cite{bt}, which describe the differential of the minimal operad of Kontsevich and Soibelman as shown in \cite{KSch} and this can be viewed as a cyclic extension of this result.  This homotopy equivalence allows us to give the explicit homotopies endowing $\op{TS}_\infty$-algebras with a homotopy BV structure.  Moreover it establishes explicitly the fact that $\op{TS}_\infty$ is a chain model for the framed little disks, i.e. that there is a zig-zag of quasi-isomorphisms:
\begin{equation}\label{zzeq}
\op{TS}_\infty\stackrel{\sim}\leftarrow\dots\stackrel{\sim}\rightarrow C_\ast(fD_2)
\end{equation}
where $C_\ast$ are the singular chains.

In the case that our field is of characteristic $0$, there is an explicit cofibrant replacement of the operad $H_\ast(fD_2)\cong \op{BV}$, called $\op{BV}_\infty$ due to \cite{GCTV}.  It is to be expected then that all $\op{TS}_\infty$-algebras are also $\op{BV}_\infty$ algebras (although not vice versa), which we establish via the following theorem (Theorem $\ref{tvthm}$).  

\begin{thmc}  In characteristic $0$ there is a quasi-isomorphism of dg operads
\begin{equation}\label{thmceq}
\op{BV}_\infty\stackrel{\sim}\to\op{TS}_\infty
\end{equation}
In particular, $\op{BV}_\infty$ satisfies the cyclic $A_\infty$ Deligne conjecture in characteristic $0$.
\end{thmc}

As noted in \cite{GCTV} the cofibrancy of $\op{BV}_\infty$ makes it a good canonical model for homotopy BV structures such as this.  It should be noted, however, that the action of $\op{BV}_\infty$ has not been made explicit; rather the morphism in equation $\ref{thmceq}$ exists by an abstract model category argument and formality.  On the other hand, the action by $\op{TS}_\infty$ will be given explicitly.  Additionally since the operad $\op{TS}_\infty$ is topological it allows us to establish the zig-zag of equation $\ref{zzeq}$ explicitly and without recourse to questions of formality of the framed little disks, and in particular allows us to prove Deligne's conjecture in any characteristic $\neq 2$. 

Finally we extend the cyclic $A_\infty$ Deligne conjecture to cyclic $A_\infty$ categories (Theorem $\ref{catthm}$). 
\begin{thmd}  Let $\op{A}$ be a cyclic $A_\infty$ category.  Then $CH^\ast(\op{A},\op{A})$ is a $\op{TS}_\infty$-algebra.
\end{thmd}

In particular $CH^\ast(\op{A},\op{A})$ is a $\op{BV}_\infty$ algebra.  

\subsection*{Future Directions}

The proof of Theorem D is achieved by a simple technique which should prove useful, beyond the scope of this paper.  There is a natural inclusion from the category of $A_\infty$ algebras $\textbf{Alg}_\infty$ to the category of (small) $A_\infty$ categories $\textbf{Cat}_\infty$ given by considering an $A_\infty$ algebra as an $A_\infty$ category with one object.  This functor has a natural (left) adjoint, as we construct below.  This adjunction is closed under cyclicity.  We then relate the Hochschild cochains of a cyclic $A_\infty$ category to the Hochschild cochains of its image under this adjunction and show that the $\op{TS}_\infty$-algebra structure on the associated cyclic $A_\infty$ algebra induces a $\op{TS}_\infty$-algebra structure on the original cyclic $A_\infty$ category.  This adjunction should provide a useful tool for the future study of $A_\infty$ categories in terms of $A_\infty$ algebras.

The most interesting context for Theorem D is the case in which $M$ is a compact oriented manifold and $N=D(T^\ast M)$ the unit disk cotangent bundle.  In this case there is an isomorphism of graded algebras between the symplectic cohomology of $N$ and the (shifted) homology of the free loop space of $M$: $SH^\ast(N)\cong H_\ast(LM)$ (see e.g. \cite{Seidel2},\cite{AS2}).  Both of these BV algebras can be modeled in Hochschild cohomology.  In \cite{CJ} the authors construct an isomorphism of graded algebras
\begin{equation*}
H_\ast(LM)\to HH^\ast(C^\ast(M),C^\ast(M))
\end{equation*}
where $C^\ast(M)$ are the singular cochains of $M$.  In parallel, on the symplectic side there is a  so called `open-closed string map':
\begin{equation*}
SH^\ast(N)\to HH^\ast(\op{F}(N),\op{F}(N))
\end{equation*}
which is known to be an isomorphism of graded algebras \cite{Seidel2}.  This puts the above objects in the following context:
\begin{equation}\label{bvdiagram}
\xymatrix{H_\ast(LM) \ar[rr]\ar[d] && HH^\ast(C^\ast(M),C^\ast(M)) \\ SH^\ast(N)\ar[u]\ar[rr] && HH^\ast(\op{F}(N),\op{F}(N))  }
\end{equation}
where all arrows are isomorphisms of graded algebras.  A direction for future research will be the following.  First the cyclic symmetry of the form of the the Fukaya category must be concretely established.  Then applying Theorem D and Theorem C it will be an immediate consequence that $HH^\ast(\op{F}(N),\op{F}(N))$ is a BV algebra, and more specifically that $CH^\ast(\op{F}(N),\op{F}(N))$ is a $\op{BV}_\infty$ algebra, or more specifically a $\op{TS}_\infty$ algebra.  It is expected, then, that an understanding of the homotopy BV structures underlying diagram $\ref{bvdiagram}$ will give insight in to the associated BV structures and the relationships between them.  For example, it is an open question whether or not the maps in this diagram are isomorphisms of BV algebras. 

Finally, if we consider $H_\ast(fD_2)\cong H_\ast(\op{M}_0)$, the homology of the moduli spaces of genus zero curves with parameterized boundary, one can look for a chain model computing the homology in all genera which also acts on the Hochschild complex of a Frobenius or cyclic $A_\infty$ structure, see e.g. \cite{TZ}, \cite{Co}, \cite{KMod1}, \cite{KMod2}, \cite{WW}.  In the Frobenius case one such chain model has been constructed via arc graphs on decorated Riemann surfaces \cite{KLP}, the so called $\op{A}rc$ operad.  In particular in \cite{KMod1} and \cite{KMod2} Kaufmann uses $\op{A}rc$ to construct a cell model of the moduli space of curves with marked points and tangent vectors at the marked points which acts on the Hochschild cochains of a Frobenius algebra.  A future direction will be to construct an $A_\infty$ version of $\op{A}rc$ operad of \cite{KLP} which acts in the cyclic $A_\infty$ algebra/category case.

\subsection*{Outline}

In Section $\ref{cyclicinftysection}$ we recall relevant details pertaining to cyclic operads and $A_\infty$ algebras and give an operadic definition of cyclic $A_\infty$ algebras.  In Section $\ref{sec:trees}$ we fix terminology for graphs and trees.  In Section $\ref{sec:cdc}$ we review Deligne's conjecture in the case of a strict Frobenius algebra and the solution of \cite{KCyclic}, emphasizing the ingredients that will be needed for our generalization.  In Section $\ref{sec:op}$ we define a new operad $\op{TS}_\infty$ and give a presentation of this operad in terms of generators, facilitating the definition of the action in Section $\ref{section:action}$.  In Section $\ref{section:underlyingtopology}$ we construct CW complexes whose cellular chains are isomorphic to the chain model $\op{TS}_\infty$ and show that $\op{TS}_\infty$ is a chain model for $fD_2$.  Piecing together our work, we give the main theorem in Section $\ref{sec:mainthm}$ and then show that the operad $\op{BV}_\infty$ of \cite{GCTV} also gives a solution to this theorem.  Finally in Section $\ref{section:cat}$ we recall cyclic $A_\infty$ categories and extend our results to them.
\tableofcontents
\begin{Acknowledgments}  I would like to thank Ralph Kaufmann for suggesting this topic and for carefully explaining his work to me.  I have also benefited from helpful discussions with Alexander Berglund, Ralph Cohen, Urs Fuchs, James McClure, Alexander Voronov, and Nathalie Wahl. 
\end{Acknowledgments}

\section{Cyclic $A_\infty$ Algebras.}\label{cyclicinftysection}  In this section we will recall the $A_\infty$ operad and give explicitly its cyclic (and unital) structure.  We will then define cyclic $A_\infty$ algebras as algebras over the cyclic unital $A_\infty$ operad.  Our emphasis will be on linear operads: operads in the category of finite dimensional vector spaces, or graded vector spaces/dg vector spaces of finite type, over a fixed field $k$ of characteristic $\neq 2$.  

\subsection{$A_\infty$ Algebras}  In order to fix notation we recall relevant details pertaining to $A_\infty$ algebras.  We denote by $K_n$ the associahedron of dimension $n-2$.  Recall the associahedron $K_n$ is an abstract polytope whose vertices correspond to full bracketings of $n$ letters and whose codimension $m$ faces correspond to partial bracketings with $m$ brackets.  Hence $K_2$ is a point, $K_3$ is an interval, $K_4$ is a pentagon.  For more detail see e.g. \cite{MSS}.  The bracketings give the collection of associahedra an operad structure induced by insertion.  As polytopes the associahedra also have a natural CW structure with $0$ cells as vertices, $1$ cells as edges, etc.  We will denote the top dimensional cell in $K_n$ by $\mu_n$ for $n\geq 2$.

\begin{definition}  The dg operad of cellular chains $\{K_n\}$ will be denoted $\op{A}_\infty$.  An algebra over this operad is an $A_\infty$ algebra.
\end{definition}

For $\op{A}_\infty$ to be an operad we must have the identity encoded in arity $1$.  As such we take as a convention that $K_1$ is a point which encodes the identity on algebras.  We will refrain from calling this point $\mu_1$, since $\mu_1$ will more commonly refer to the differential on algebras.

\begin{remark} \label{symmetrization}  The associahedra form a non-$\Sigma$ operad.  The termwise tensor product with the associative operad $\op{A}s$ (see example $\ref{Frobeniusalgebra}$ below) gives a functor $-\tensor \op{A}s$ taking linear non-$\Sigma$ operads to linear operads by taking the symmetric group action only on the $\op{A}s$ factor.  We will adopt the terminology of \cite{MSS} and call the image of this functor the symmetrization of the input.  In what follows we can work in both the symmetrized or unsymmetrized versions by choosing to label or not label the e.g edges of polygons or vertices of trees.
\end{remark}

\subsection{The Unital Structure of $\op{A}_\infty$.}

\begin{definition}  Let $\op{O}$ be a linear operad and define $\op{O}(0)=k$.  We say $\op{O}$ is an operad with unital multiplication $\mu_2$ if the usual operad composition rules (which hold for $n\geq 1$) can be extended to hold for $n\geq 0$ in such a way that for $i=1,2$,
\begin{equation} \label{unitalequation}
\op{O}(2)\tensor\op{O}(0)\stackrel{\circ_i}\longrightarrow\op{O}(1)
\end{equation}
takes $\mu_2 \mapsto \nu \in Hom(k, \op{O}(1))$, where $\nu$ is the operadic unit.
\end{definition}
This definition is a nonassociative version of the notion of operad with multiplication in \cite{MS}.  The extra structure maps involving $\op{O}(0)$ will be called degeneracies.

\begin{example}  
Let $A$ be an algebra with multiplication $\ast$ and let $u$ be a unit with respect to the multiplication.  Define the degeneracies
\begin{equation*}
End_A(n)\tensor End_A(0) \stackrel{\circ_i}\longrightarrow End_A(n-1)
\end{equation*}
to be insertion of $u$ into the $i^{th}$ argument of a function.  Then $End_A$ is a unital operad with multiplication.  This example illustrates why we require equation $\ref{unitalequation}$ to hold.  If we did not require this then we could define degeneracies by inserting any element, not just a unit.
\end{example}

\begin{definition} Let $A$ be an algebra with a unital multiplication $\ast$.  We say $A$ is a unital algebra over the unital operad with multiplication $\op{O}$ if there is a morphism of operads
\begin{equation*}
\op{O}\to End_A
\end{equation*}
taking $\mu$ to $\ast$ and taking degeneracies to degeneracies.
\end{definition}

\begin{example}  Define degeneracies on the $\op{A}_\infty$ operad by taking for $n\geq 3$
\begin{equation*}
\mu_n\circ_i 1=0
\end{equation*}
and for $n=2$
\begin{equation*}
\mu_2\circ_i 1= 1
\end{equation*}
In other words the degeneracies are only non zero when considered on the suboperad generated by $\mu_2$.  Then $\op{A}_\infty$ is an operad with unital multiplication.  For degree reasons, this is the only unital structure on this operad.  A unital algebra over this operad with unital multiplication will be called a unital $A_\infty$ algebra.  
\end{example}
\begin{remark}\label{unitremark}
For the remainder of this paper we will consider only the unital version of the $A_\infty$ operad and it's algebras. For practical purposes this is equivalent to postulating 
\begin{enumerate}
\item{}  All $A_\infty$ algebras we consider have a unit with respect to their binary multiplication.
\item{}  For $n\geq 3$, $\mu_n$ is a `normalized cochain', meaning $\mu_n(a_1\tensor\dots\tensor a_n)=0$ if $a_i$ is the unit element for some $i$. 
\end{enumerate}
\end{remark}

\subsection{Hochschild Cohomology of an $A_\infty$ Algebra.}  We will now define the Hochschild cohomology of an $A_\infty$ algebra following the presentation in \cite{KS}.  For a graded linear operad $\op{O}$ we can associate an odd-Lie algebra, $(\op{O}_\ast, [-,-])$ to $\op{O}$ as follows.  Let $\op{O}_\ast:=\bigoplus\op{O}(n)$ and take the vector space $\op{O}_\ast$ to be graded by the total grading, i.e. if $a \in \op{O}(n)$ is an element of degree $deg(a)$ then we consider the degree of $a \in \op{O}_\ast$ to be the total degree: $||a||=deg(a)+n$.  Then define for $a \in\op{O}(n)$ and $b \in \op{O}(m)$;
\begin{equation*}
[a,b] := a\circ b - (-1)^{(||a||-1)(||b||-1)}b\circ a
\end{equation*}
where 
\begin{equation}\label{preliesigns}
a \circ b:= \ds\sum_{i=1}^n (-1)^{(i-1)(m-1)+(n+1)deg(b)} a\circ_i b
\end{equation}
Notice that by assumption the operad associativity holds with respect to the original grading.  The $(n+1)deg(b)$ appearing in the sign must be included to assure that $a\circ b$ will satisfy the odd pre-Lie identity with respect to the $\textit{total grading}$.  Consequently $(\op{O}_\ast, [-,-])$ is an odd-Lie algebra.  This construction is due to Gerstenhaber \cite{G} and the graded sign appears in e.g. (\cite{LV} p. 293).

\begin{proposition} \label{multdiff}
Let $\op{O}$ be an operad with associative multiplication $\mu_2 \in \op{O}$ (of degree zero).  Then 
\begin{equation*}
d(a):=[a,\mu_2]
\end{equation*}
defines a square zero differential on $\op{O}_\ast$ of degree $+1$.
\end{proposition}
\begin{proof}  By the odd Jacobi identity
\begin{eqnarray*}
d^2(a)&=&[[a,\mu_2],\mu_2] \\ \nonumber &=&[a,[\mu_2,\mu_2]]+(-1)^{(||\mu_2||-1)^2}[[a,\mu_2],\mu_2]
\end{eqnarray*}
Now, $||\mu_2||=2$ thus $\mu_2\circ\mu_2=0$ by associativity, and so $[\mu_2,\mu_2]=0$.  Thus $d^2(a)=-d^2(a)$, hence $d^2=0$.  Considering the degree, for $a\in\op{P}(n)$ we have 
\begin{equation*}
||d(a)||=||a\circ\mu_2\pm\mu_2\circ a|| = deg(a\circ\mu_2\pm\mu_2\circ a)+n+1= deg(\mu_2)+deg(a)+n+1 =||a||+1  
\end{equation*}
\end{proof}

\begin{example}  Let $A$ be an associative algebra.  Then the endomorphism operad $End_A$ is an operad with associative multiplication.  Thus we can consider the differential $d:End_A(n)\to End_A(n+1)$.  This differential is precisely the differential of Hochschild cohomology.
\end{example}

We can generalize Proposition $\ref{multdiff}$ as follows.
\begin{definition}
Let $\op{O}$ be an operad and let $\zeta\in \op{O}_\ast$ be an element of even degree with respect to the inherited $\Z/2$ grading satisfying the equation $[\zeta,\zeta]=0$.  Then we define
\begin{equation*}
d_\zeta(a):=[a,\zeta]
\end{equation*}
By the same argument as Proposition $\ref{multdiff}$ we have $d_\zeta^2=0$.  As such $(\op{O}_\ast, d_\zeta)$ is a cochain complex.  
\end{definition}

\begin{example}  Let $A$ be an $A_\infty$ algebra.  Let $\mu_1$ denote the differential on $A$ and define
\begin{equation}\label{mueq}
\mu:=\ds\bigoplus_{n\geq 1}\mu_n  
\end{equation}
Notice that $||\mu_n||=(n-2)+n$, hence $\mu$ is of even degree with respect to the inherited $\Z/2$ grading.  Now for any $A_\infty$ algebra and natural number $t$ we have that
\begin{equation*}
\bigoplus_{r+s=t+1}\mu_r\circ\mu_s = 0
\end{equation*}
and hence $\mu\circ\mu=0$.  As such $[\mu,\mu]=0$ and so $d_\mu:=[-,\mu]$ is a differential of total degree $+1$.
We then define the Hochschild cohomology of $A$ to be the cohomology of the cochain complex $(End_A, d_\mu)$ and write $HH$ (resp $CH$) for the cohomology (resp. cochains).
\end{example}
\begin{notation}\label{degreenotation}  We fix degree notation for Hochschild cochains as follows.  For $f \in CH^n(A,A)$ we write $deg(f)$ for the degree of $f$ with respect to $A$, we write $|f|=n$, the number of inputs, and we write $||f||$ for the total degree:  
\begin{equation*}
||f||=deg(f)+|f|
\end{equation*}
\end{notation}

\begin{remark}\label{gradingconventions}  The grading we have chosen for $\op{O}_\ast$ is a convention.  Other advantageous conventions would be for $a\in\op{P}(n)$ to take $deg(a)-n$ or $deg(a)-n+2$ as the total degree.  The disadvantage of these conventions would be that the differential would have degree $-1$, which does not agree with the standard convention for grading in the Hochschild complex.  See also Remark $\ref{gradingconventions2}$.
\end{remark}

\subsection{Cyclic Operads: Definition}
Since we will make explicit use of the axioms for cyclic operads we recall the definition here.  Let $S_n^+$ be the group of permutations of the set $[n]:=\{0,\dots,n\}$.  We view the symmetric group $S_n$ as a subset of $S_n^+$.  Let $\tau_n$ be the permutation $(0 \dots n)\in S_n^+$.
\begin{definition}\label{cycdef} \cite{GeK1}  Let $\op{P}$ be an operad in a symmetric monoidal category $(\op{C}, \iota, \tensor)$ and suppose there is an action of $S_n^+$ on each $\op{P}(n)$ such that:
\begin{itemize}
\item  The action of $S_n^+$ restricted to the subset $S_n$ agrees with the underlying operad structure.
\item  If $\eta \colon \iota \to \op{P}(1)$ is the identity map then $\tau_1 \circ \eta = \eta$.
\end{itemize}
We say $\op{P}$ is cyclic if the following diagram commutes for  $1\leq i \leq m-1$

\begin{equation}\label{cycdefeq1}
\xymatrix{\op{P}(n)\tensor \op{P}(m) \ar[rr]^{\circ_i} \ar[d]_{\tau_n\tensor id} && \op{P}(n+m-1) \ar[d]^{\tau_{n+m-1}} \\ \op{P}(n)\tensor\op{P}(m) \ar[rr]^{\circ_{i+1}} && \op{P}(n+m-1)}
\end{equation}
and the following diagram commutes for $i=m$.

\begin{equation}\label{cycdefeq2}
\xymatrix{\op{P}(n)\tensor \op{P}(m) \ar[rr]^{\circ_m} \ar[d]_{\tau_n\tensor\tau_m}   && \op{P}(n+m-1) \ar[dd]^{\tau_{n+m-1}} \\ \op{P}(n)\tensor \op{P}(m)  \ar[d]_{s_{\op{P}(n)\tensor\op{P}(m)}} &  \\ \op{P}(n)\tensor\op{P}(m) \ar[rr]^{\circ_1} && \op{P}(n+m-1)}
\end{equation}
where $s_{\op{P}(m),\op{P}(n)}$ interchanges the order of the monoidal product $\op{P}(m) \tensor \op{P}(n) \to \op{P}(n) \tensor  \op{P}(m)$.  
\end{definition}
There is also a notion of a non-$\Sigma$ cyclic operad (\cite{MSS} p. 257) which we will need.
\begin{definition} \label{nonsigmacycdef}  A non-$\Sigma$ operad $\op{P}$ is said to be a non-$\Sigma$ cyclic operad if it comes equipped with an action of $\Z_{n+1}$ on $\op{P}(n)$ for each $n$ satisfying the conditions of Definition $\ref{cycdef}$ (where we consider $\Z_{n+1}\subset S_n^+$ under the identification $1\sim\tau_n$).  
\end{definition}

We can talk about cyclic operads in any symmetric monoidal category, but we again restrict our primary attention to the linear case.  The endomorphism operad $End_A$ of a vector space $A$ is not $\textit{a priori}$ cyclic, but additional structure on $A$ can ensure that it is. 

\begin{definition} \label{cyclicalgebra}  We call a vector space $A$ cyclic if it comes with a symmetric nondegenerate bilinear form.  If $A$ is dg we further require
\begin{equation*}
\langle d(a),b \rangle= (-1)^{|b|-1}\langle a, d(b) \rangle 
\end{equation*}
\end{definition}
This terminology is motivated by the following fact.
\begin{lemma}  If $A$ is cyclic then $End_A$ is a cyclic operad.
\end{lemma}
\begin{proof}  Note that a cyclic vector space is canonically self dual by the map
\begin{equation*}
v\mapsto[w\stackrel{\phi_v}\mapsto\langle v,w\rangle]
\end{equation*}
This self duality of $A$ gives us an isomorphism
\begin{equation*}
End_A(n) \cong (A^\ast)^{\tensor n+1}
\end{equation*}
which gives us the $S_{n+1}$ action.
\end{proof}

A morphism of cyclic operads is a morphism of operads between two cyclic operads that is $S_n^+$ equivariant.  To say $A$ is an algebra over a cyclic operad $\op{P}$ means that $A$ is cyclic and there is a morphism of cyclic operads,
\begin{equation*}
\op{P} \to End_A
\end{equation*}  
In this case we will also use the terminology cyclic $\op{P}$-algebra, or just cyclic algebra if the operad is clear. 
\begin{example}\label{Frobeniusalgebra}  Let $\op{A}s$ be the operad for associative algebras.  That is, $\op{A}s(n)$ is a $k$ vector space of dimension $n!$ with a basis indexed by unparenthesized sequences of letters indexed by $\{1\cdc n\}$ in all possible orders.  The operad structure map $\circ_i$ is given by insertion of one sequence of letters in to the $i^{th}$ position of the other, and the $S_n$ action is given by permuting the letters of a sequence.  An algebra over this operad has the structure of an associative $k$ algebra with multiplication parameterized by the sequence $x_1x_2 \in \op{A}s(2)$ and its image in the endomorphism operad will be called $\mu_2$.  The operad $\op{A}s$ is cyclic by taking the action of $(0\dots n)$ to be the identity.  A cyclic algebra $A$ over this cyclic operad has a symmetric nondegenerate bilinear form which satisfies
\begin{equation*}
\langle a,bc \rangle = \langle ab,c \rangle \ \ \forall  \ a,b,c  \in A
\end{equation*}
Because 
\begin{equation*}
\langle a,\mu_2(b,c) \rangle = \langle c,\mu_2(a,b) \rangle
\end{equation*}
holds due to the compatibility of the cyclic structures and the morphism of operads.  So we see that a cyclic $\op{A}s$-algebra is a (not necessarily commutative) Frobenius algebra.
\end{example}

\subsection{The Cyclic Structure of $\op{A}_\infty$}  As pointed out in \cite{GeK1}, the number of vertices of the associahedron $K_n$ is equal to the number of triangulations of the regular polygon with $n+1$ vertices.  We now give an explicit dual graph construction which fixes an assignment that endows $\op{A}_\infty$ with a non-$\Sigma$ cyclic structure.

First suppose $n\geq 3$ and let $P_n$ be the regular polygon with $n$ vertices and with a distinguished edge called the base.  Realize $P_n$ in the complex plane such that the based edge has vertices at $0$ and $1$ and with all other vertices above the real line.  By a partial triangulation of $P_n$ of degree $m$, where $0\leq m \leq n-3$ we mean a triangulation that is missing $m$ edges.  Given a partial triangulation we form a dual planar tree as follows.  The vertices of the tree are the midpoints of the edges of the partial triangulation.  The vertex at $1/2$ will be called $v_0$.  We say two vertices are adjacent if they border a region enclosed by, but not intersecting, the edges of the triangulation.  We then assign a height to each tree vertex inductively as follows.  The vertex $v_0$ has height $0$.  The vertices adjacent to $v_0$ have height $1$.  The vertices adjacent to those of height $1$ have height $2$ (except $v_0$ whose height was already assigned), and so on.  Two vertices are joined by an edge if they are adjacent and have different heights.  It is plain to see that the result is a planar tree where the tree height is given by the vertex height, and the number of internal edges of the tree is equal to the number of internal edges of the triangulation.  In particular a full triangulation will have a binary tree.  Since planar trees with $n$ leaves correspond in the obvious way to partial parenthisizations of $n$ letters, this gives the assignment to the cells of associahedra.  See Figure $\ref{tri2fig}$.  

\begin{figure}[ht]
\centering
\includegraphics[scale=.95]{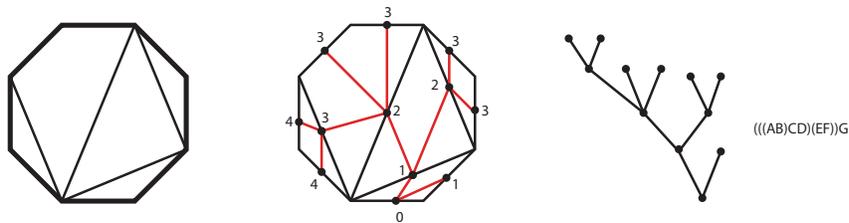}
\caption{A degree 1 partial triangulation of $P_8$ yields a tree corresponding to a cell in $CC_1(K_7)$ }
\label{tri2fig}
\end{figure}

Now $P_{n+1}$ has a $\Z_{n+1}$ action given by rotating in the clockwise direction about the center by $2\pi/{(n+1)}$.  This gives a $\Z_{n+1}$ action to the set of partial triangulations of $P_{n+1}$ and hence to $CC_\ast(K_n)=\op{A}_\infty(n)$.  See Figure $\ref{fig:rotate}$ for an example.  Notice that for each $n$ the action on the top dimensional cell in $K_n$ is trivial modulo sign, since it corresponds to the rotation of a polygon with no triangulating edges.  We will determine below that the signs on top dimensional cells are given by $\tau_n(\mu_n)=(-1)^n \mu_n$.  This defines the action for $n\geq 3$.  We then define the action $\tau_1$ and $\tau_2$ on the associahedra $K_1$ and $K_2$ to be the identity.  Notice that this is not the same as extending $\tau_n(\mu_n)=(-1)^n \mu_n$ to hold for all $n$.  Doing this would give us an anticyclic operad and the notion of a symplectic $A_\infty$ algebra.

\begin{figure}
\centering
\includegraphics[scale=.48]{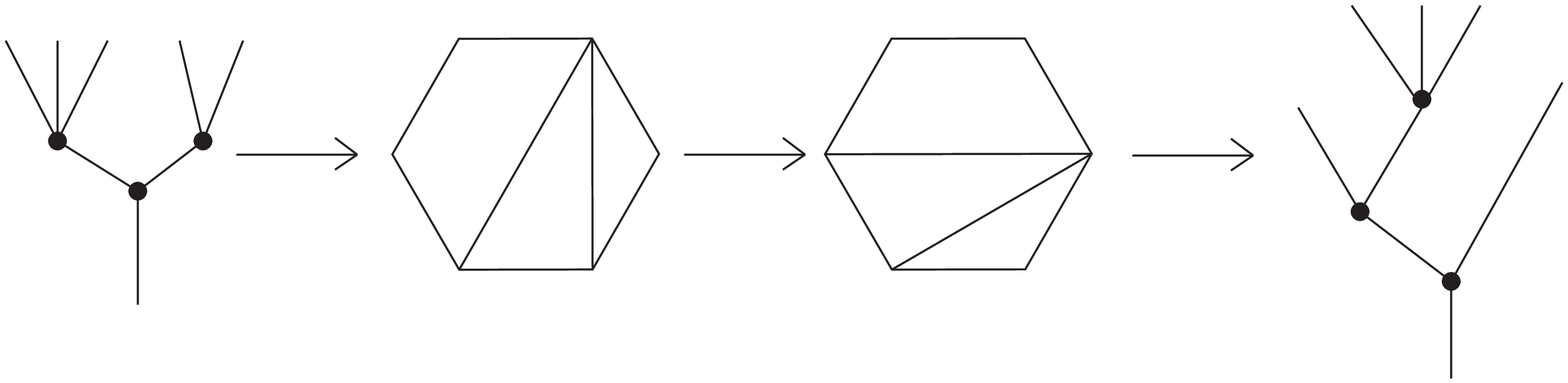}
\caption{An example of the $\Z_{n+1}$ action on $\op{A}_\infty(5)$: $\tau_5((\mu_2\circ_2\mu_2)\circ_1\mu_3) = \pm\mu_2\circ_1(\mu_2\circ_2\mu_3)$}
\label{fig:rotate}
\end{figure}

\begin{remark} \label{polygonremark}  The one to one correspondence between binary trees with $n-1$ leaves and triangulations of an $n$-gon is well known.  The above dual graph construction, and in particular its relation to cyclic operads, does not appear in the literature as far as I am aware.  This gives a new geometric interpretation to the $A_\infty$ operad whose composition maps we can now describe by attaching regular polygons along their faces and then reshaping the result to be again a regular polygon. 
\end{remark}

\begin{lemma}  This assignment makes $\op{A}_\infty$ a non-$\Sigma$ cyclic operad.
\end{lemma}
\begin{proof}  This amounts to checking that the action of $\tau_n$ on $\op{A}_\infty(n)$ is compatible with the identity and satisfies diagrams $\ref{cycdefeq1}$ and $\ref{cycdefeq2}$.  Compatibility with the identity follows by the definition $\tau_1\equiv id$.  To see that diagrams $\ref{cycdefeq1}$ and $\ref{cycdefeq2}$ hold is easiest if we use the model for $A_\infty$ suggested in Remark $\ref{polygonremark}$.  In either diagram one can glue then rotate or rotate then glue.  It is clear that the edges which are glued and the edge which acts as the base are independent of this choice and hence the diagrams commute.    

\end{proof}
\begin{remark}
The triangulated polygon framework can be modified so that the planar planted trees with height of \cite{KSch} are dual.  Do this by defining weighted triangulations of polygons where each internal edge carries a weight in $(0,1]$.  Taking chains on this topological operad gives a cubical decomposition of associahedra which is equivalent to the Boardman Vogt $W$ construction on the trivial non-$\Sigma$ operad \cite{MSS}.  Rotating these weighted triangulations gives the cubical decomposition of associahedra a non-$\Sigma$ cyclic operad structure.
\end{remark}

\begin{remark}  The above $\Z_{n+1}$ action makes $\op{A}_\infty$ a non-$\Sigma$ cyclic operad.  We can then apply the functor $-\tensor\op{A}s$, as in Remark $\ref{symmetrization}$.  Thinking of $\op{A}s$ as a cyclic operad the symmetrization is cyclic as well.  When we consider $\op{A}_\infty$ as a cyclic operad, it should be clear from the context if we mean the non-$\Sigma$ version or the symmetrization.
\end{remark}

\begin{proposition}  The non-$\Sigma$ cyclic structure on $\op{A}_\infty$ given above is unique.
\end{proposition}
\begin{proof}  The operad $\op{A}_\infty$ is generated by top dimensional cells $\mu_n$ under operadic composition.  For degree reasons and since $\tau_n^{n+1}=1$, we must have $\tau_n(\mu_n)=\zeta_{n+1}\mu_n$ where $\zeta_{n+1}$ is some $(n+1)st$ root of unity.  It is then enough to show that the differential dictates that $\zeta_{n}=(-1)^{n+1}$ for $n\geq3$.  To see this we proceed by induction.  First calculate (using axioms in $\ref{cycdef}$):  

\begin{eqnarray}\label{signeqs}
\zeta_{4}d(\mu_3)&=&\tau_3(d(\mu_3))=\tau_3(\mu_2\circ_1\mu_2-\mu_2\circ_2\mu_2)\\ \nonumber
& = & \tau_2(\mu_2)\circ_2\mu_2-(-1)^{|\mu_2||\mu_2|}\tau_2(\mu_2)\circ_2\tau_1(\mu_2) \\ \nonumber
& = & \zeta_3\mu_2\circ_2\mu_2-\zeta_3^2\mu_2\circ_1\mu_2 = \zeta_4(\mu_2\circ_1\mu_2-\mu_2\circ_2\mu_2)
\end{eqnarray}
and so $\zeta_4 =-\zeta_3=-\zeta_3^2$, hence $\zeta_3=1$ and $\zeta_4=-1$.  For the induction step notice that since $d(\mu_n)$ can be written as a signed sum of compositions of lower $\mu_i$ (see e.g. \cite{MSS} p. 195) we can calculate $\zeta_{n+1}$ uniquely in terms of the lower $\zeta_i$ as above.
\end{proof}
\subsection{Cyclic $A_\infty$ Algebras} \label{casection} 
We can now give our operad theoretic definition of cyclic $\op{A}_\infty$ algebras.

\begin{definition} \label{cycinftydef} Let $A$ be a dg vector space.  We say $A$ is a cyclic $A_\infty$ algebra if $A$ is a cyclic unital algebra over the cyclic unital operad $\op{A}_\infty$.
\end{definition}
  
\begin{proposition}\label{cycalgprop} If $A$ is a cyclic $A_\infty$ algebra then $A$ is a unital $A_\infty$ algebra equipped with a symmetric nondegenerate bilinear form $\langle -,- \rangle$ such that $\forall \ a_0,\dots,a_{n} \in A$, 
\begin{equation}\label{eqcyc}
\langle a_0, \mu_n(a_1 \tensor a_2\tensor \dots \tensor a_{n}) \rangle = (-1)^{n+|a_0|(|a_1|+\dots+|a_n|)}\langle a_{n},\mu_n(a_0\tensor  a_1\tensor  \dots \tensor a_{n-1}) \rangle \  
\end{equation}
for each $n \geq 2$.
\end{proposition}
\begin{proof}  Say $A$ is a cyclic $A_\infty$ algebra and let $\rho \colon \op{A}_\infty \to End_A$ be the relevant morphism of cyclic operads.  The map
\begin{equation}\label{invareq}
\op{A}_\infty(n)\tensor A^{\tensor n+1} \to k
\end{equation}
given by
\begin{equation*}
 f\tensor a_0\tensor\dots\tensor a_n  \mapsto \langle a_0,f(a_1\tensor\dots\tensor a_n) \rangle
\end{equation*}
is invariant under the simultaneous $S_n^+$ action on the left hand side of equation $\ref{invareq}$ (\cite{MSS} proposition 5.14).  Hence 
\begin{eqnarray*}
\langle a_0,\mu_n(a_1\tensor\dots\tensor a_n) \rangle & = & (-1)^{|a_0|(|a_n|+\dots+|a_1|)}\langle a_n,\tau_n(\mu_n)(a_0\tensor\dots\tensor a_{n-1}) \rangle \nonumber \\ \nonumber & = & (-1)^{n+|a_0|(|a_n|+\dots+|a_1|)}\langle a_n,\mu_n(a_0\tensor\dots\tensor a_{n-1}) \rangle
\end{eqnarray*}
\end{proof}
\begin{remark}\label{invardiffrmk}  The above proposition also holds when $n=1$ if we define $\mu_1$ to be the differential on $A$ due to definition $\ref{cyclicalgebra}$.  In particular we can calculate
\begin{equation*}
\langle a, d_A(b) \rangle = (-1)^{|b|-1}\langle d_A(a),b \rangle = (-1)^{|b|-1+|b|(|a|+1)}\langle b,d_A(a) \rangle=(-1)^{1+|a||b|} \langle b,d_A(a) \rangle
\end{equation*}
Notice the proposition would not hold at $n=1$ if we had consider $\mu_1$ to be the identity.
\end{remark}
Notice that the cyclic structure of $\op{A}_\infty$ gives us a notion of compatibility in the context of more complicated cells.  For example from Figure $\ref{fig:rotate}$ we see that,
\begin{equation*}
\langle a_0, \tau_5((\mu_2\circ_2\mu_2)\circ_1\mu_3)(a_1 \tensor\dots\tensor  a_5) \rangle =\pm\langle a_5, \mu_2\circ_1(\mu_2\circ_2\mu_3)(a_0\tensor  a_1\tensor  \dots \tensor a_4) \rangle
\end{equation*}
Hence, operads are the ideal tool to describe such an algebra; the infinitely many axioms needed to define invariance of the bilinear form are encoded by the cyclic operad.

\section{Trees}\label{sec:trees}  In this section we will fix notation and definitions pertaining to trees.  This presentation partly follows \cite{KCyclic}.  

\begin{definition}  A graph $\Gamma$ is a quadruple $(F(\Gamma), V(\Gamma), \lambda_\Gamma, \iota_\Gamma)=(F,V,\lambda,\iota)$ where $F$ and $V$ are finite sets, $\lambda$ is a map $F\to V$ and $\iota$ is a map $F \to F$ such that $\iota^2$ is the identity.
\end{definition}
We will use the following terminology with respect to graphs.
\begin{itemize}
\item{}  The elements of $V$ are called the \textbf{vertices} of $\Gamma$.  
\item{}  The elements of $F$ are called the \textbf{flags} of $\Gamma$.
\item{}  For $v\in V$, the elements of $\lambda^{-1}(v)$ are called the $\textbf{flags at $v$}$.
\item{}  For $v\in V$, the $\textbf{valence}$ of $v$ is $|\lambda^{-1}(v)|$ and is denoted $val(v)$.
\item{}  An $\textbf{edge}$ of $\Gamma$ is a pair $(f, \iota(f))$ such that $f \in F$ and $f\neq\iota(f)$.
\item{}  The $\textbf{tails}$ of $\Gamma$ are those flags $f$ such that $f=\iota(f)$.
\item{}  An $\textbf{isomorphism of graphs}$ is a bijection of the flags and vertices that preserves the defining maps ($\lambda$ and $\iota$).  From now on we do not distinguish between isomorphic graphs.
\end{itemize} Given a graph $\Gamma$ with an edge $e$ we can contract $e$ and get a new graph as follows.

\begin{definition}\label{edgecontraction}  Let $\Gamma$ be a graph with an edge $e=(f,\iota(f))$ and with $\lambda(f)=v, \lambda(\iota(f))=w$.  Define a new graph $\Gamma/e$ called the contraction of $e$ in $\Gamma$ by $F(\Gamma/e):=F(\Gamma)\setminus\{f,\iota(f)\}, V(\Gamma/e):=(V(\Gamma)\setminus \{v,w\})\cup (v\wedge w)$ (where $v\wedge w$ is an arbitrary new element of the vertex set), $\iota_{\Gamma/e}:=\iota_\Gamma $ and finally for a flag $g \in F(\Gamma)$ we define
\begin{equation*}
\lambda_{\Gamma/e}(g):= \begin{cases} \lambda_{\Gamma}(g) & \text{ if } \lambda_{\Gamma}(g)\not\in\{v,w\} \\ v\wedge w & \text{ if } \lambda_\Gamma(g)\in \{v,w\}  \end{cases}
\end{equation*}
\end{definition}

To a graph $\Gamma$ we can associate a $1$-dimensional CW complex in an obvious way.  Namely, the open $1$-cells correspond to the collection of edges and tails and the $0$ cells to the vertices of $\Gamma$ and the closure of the tails.  In particular the boundary relation for an edge $(f,\iota(f))$ is
\begin{equation*}
\partial[(f,\iota(f))]= \lambda(\iota(f))-\lambda(f)
\end{equation*}

\begin{definition}  A graph will be called a tree if the realization of the associated CW complex is connected and contractible.
\end{definition}
We will use the following terminology with respect to trees.
\begin{itemize}
\item{}  A tree $\tau$ together with a choice of distinguished tail $f$ will be called a $\textbf{rooted tree}$ with root $f$.
\item{}  A $\textbf{planar tree}$ is a tree with a cyclic order on the flags at each vertex.
\item{}  A $\textbf{planar planted tree}$ is a rooted planar tree together with a linear order on the set of flags at the root such that the root comes first in the linear order associated to its vertex.  This order will be called the $\textbf{planar order}$ of the flags at a vertex.  The flag coming first in the planar order will be called the $\textbf{outgoing flag}$ at $v$.  The remaining flags will be called the $\textbf{incoming flags}$ of $v$.
\item{}  If $\tau$ is a planar planted tree with a vertex $v$, the $\textbf{branches}$ of $\tau$ at $v$, denoted $br_v(\tau)$, are the connected components of the graph formed by deleting the vertex $v$ and any non-root flags $f$ having $\lambda(f)=v$.  Notice these components are rooted trees (with one exception) taking roots $\iota(f)$ for $f$ an incoming flag at $v$ and the original root on the component corresponding to the outgoing flag at $v$.  The exception occurs if the root of $\tau$ is adjacent to $v$.  In this case we consider the root as a branch of $\tau$ at $v$, although it is not technically a tree.  Notice that the branches at $v$ have a natural linear order coming from the planar order of $\tau$.
\item{}  The $\textbf{arity}$ of a vertex $v$ in a tree is the number of incoming flags at $v$ and is denoted $|v|$.  Note that $val(v)=|v|+1$.  
\item{}  A $\textbf{leaf}$ of a tree is a vertex whose only flag is outgoing.  A $\textbf{corolla}$ is a tree with only one non-leaf vertex.
\item{}  A tree will be called a $\textbf{black and white tree}$ (or b/w) if it comes equipped with a map $clr\colon V \to \Z/2$.  Those vertices mapped to $1$ are called $\textbf{white vertices}$ and those mapped to $0$ are called $\textbf{black vertices}$.  An edge $(f,\iota(f))$ such that $clr(\lambda(f))=clr(\lambda(\iota(f)))=1$ (resp. $0$) will be called a $\textbf{white (resp. black) edge}$.
\item{}  A black and white tree will be called $\textbf{bipartite}$ if for each edge $(f,\iota(f))$, the vertices $clr(\lambda(f))\neq clr(\lambda(\iota(f)))$.
\end{itemize}
I will remark here that abstractly there is no difference between a black and white vertex, but these two types of vertices will play a very different role in what follows.

\subsection{Drawing Trees}  Let $\tau$ be a planar planted tree with root $f$.  We depict $\tau$ graphically as follows.  The vertex $v_0=\lambda(f)$ is placed lower than all other vertices.  The non-root flags $f_i$ at $v_0$ are depicted as line segments drawn at angles $\theta_i$ in $(0,\pi)$ such that $f_l< f_j \Rightarrow \theta_l> \theta_j$.  If a flag $f$ belongs to an edge we put the vertex $\iota(f)$ at the top of the associated line segment.\footnote{Since realizations of trees are defined only up to homeomorphism, there is no need to make edges twice the length of the flags.}  We then continue in this manner, where the outgoing flag of a vertex is drawn below the vertex and the incoming flags are drawn above such that the planar order goes from left to right, until all flags and vertices are depicted.  Vertices are depicted as small circles, and if the tree is b/w we depict the black vertices as filled in and the white vertices as hollow.  Finally we attach a line segment (representing the root $f$) to the vertex $v_0$ pointing straight down and place a small square at the end of this line segment.  The square is not a vertex of $\tau$.  Given a planar planted tree there is a unique way to draw such a picture and given a picture as drawn above there is a unique planar planted tree that can be associated to it in the coherent way.  As such we no longer distinguish between a planar planted tree and its pictorial representation.

\subsection{Trees With Spines}  For each positive integer $n$ there is a cellular decomposition of $S^1$ with $n$ $0$-cells corresponding to $e^{2\pi \sqrt{-1}j/n}$ for $j=1\cdc n$ and taking $1$-cells corresponding to pieces of the unit circle connecting two adjacent vertices.  We will call the set of all cells given by this decomposition $C^n_\ast(S^1)$.  We call the $0$-cell associated to $1 \in \C$ the `base point'.
\begin{definition}\label{treeswithspines}
A tree with spines is a bipartite planar planted tree whose only tail is the root, along with a choice of cell $\eta(v)\in C^{val(v)}_\ast(S^1)$ for each white vertex $v$.  We call this choice of cell the spine of $v$.
\end{definition}
We depict a tree with spines graphically as follows.  Consider each white vertex $v$ as a $S^1$ rotated by $\pi/2$ in the clockwise direction so that the base point coincides with the outgoing flag of the vertex.  Then the $0$-cells in $C^{val(v)}_\ast(S^1)$ correspond to the flags at $v$ and the $1$-cells correspond to the portion of the circle $v$ between two adjacent flags.  If the cell associated to $v$ is a $1$-cell we place a tic mark on the associated portion of the vertex.  If the cell associated to $v$ is a $0$ cell we place a tic mark on the associated flag, unless the associated flag is outgoing, in which case we draw no tic mark.  We call a tree with spines spineless if the cell associated to each vertex is the base point.  Note that a tree that is not a tree with spines can be considered as a tree with spines by taking it to be spineless.  As such we tacitly assume that all bipartite planar planted trees are trees with spines from now on.

\subsubsection{Contraction of a white angle}  Let $\tau$ be a tree with spines and $v$ a white vertex of $\tau$.  We refer to the arcs of the circle $v$ between two adjacent flags as the white angles at $v$ of $\tau$.  The set of white angles at all white vertices of $\tau$ will be called simply the white angles of $\tau$.  A white angle corresponding to the spine of $v$ will be called spined, otherwise it will be called non-spined.

\begin{definition}  Let $\tau$ be a tree with spines and let $\theta_v$ be a non-spined white angle at a white vertex $v$ of $\tau$.  We define a new tree with spines $\tau/\theta_v$ as follows.  If neither of the two flags adjacent to $\theta_v$ is the root then collapsing this angle combines the two associated edges in to one edge (smashing the two black vertices in to one).  If the spine at said white vertex was on one of these two flags then the new amalgamated edge has the spine.  On the other hand if one of the flags was the root flag then the root is now attached to the black vertex at the end of the edge associated to the other flag.  We call the map of trees with spines given by $\tau \mapsto \tau/\theta_v$ the contraction of the white angle $\theta_v$.
\end{definition}

\subsection{Marked Trees}\label{section:markedtrees} Let $\tau$ be a tree with spines and $v$ a black vertex of $\tau$ with outgoing flag $f$.  We define $||v||=|v|+1$ if $\lambda(\iota(f))$ is a white vertex whose spine corresponds to $\iota(f)$ and define $||v||=|v|$ otherwise.
\begin{definition}\label{markedtrees}
A marked tree is a tree with spines along with a labeling of each black vertex by a cell in $K_{||v||}\times S_n$.
\end{definition}
The first observation concerning marked trees is that the set of flags at a black vertex $v$ in a marked tree has two linear orders.  One comes from the planar order and takes the outgoing edge first and then the incoming edges left to right.  The other comes from the label $\gamma\times\sigma$ and will be called the label order.  Suppose the label is of the form $\gamma\times 1$.  Then if $||v||=|v|$ the two orders agree.  If $||v||=|v|+1$ then the cyclic orders agree but we take the linear order as starting at the first incoming flag in the planar order.  If $\sigma \neq 1$ then we simply permute the label order of $\gamma\times 1$ by $\sigma$.  In other words, the label $\gamma\times\sigma$ puts a label of $1\cdc n$ on the incoming (resp. all) flags at $v$ for $||v||=|v|$ (resp. $||v||=|v|+1$) which need not in general agree with the planar order.

We can depict a marked tree $\tau$ graphically in two ways.  The first way is to draw the tree as above and to write a label next to the depiction of the associated black vertex.  The other is to depict the black vertex as the tree (drawn with all black vertices) corresponding to the label.  It is important to remember that in such a depiction the actual tree is the result of contracting the black edges in the picture one at a time (see Definition $\ref{edgecontraction}$), each time relabeling the new amalgamated black vertex as the operadic composition of the two labels of its predecessors.  Note that the associativity axiom for operads guarantees that the order of the contractions is immaterial.  Since a marked tree is always bipartite, no confusion should result if we use the second depiction of a marked tree.   

\begin{definition}\label{trivialvertex}  A black vertex $v$ in a marked tree will be called trivial if $||v||=|v|=1$.
\end{definition}
Notice that trivial vertices are labeled with the identity in $K_1$ by definition.

\begin{remark}  We have made the choice to define trees with spines and marked trees as bipartite.  The primary reason for this is that we do not wish to distinguish between a tree with white edges and the bipartite tree formed by placing trivial black vertices in the middle of all white edges.  Restricting our attention to bipartite trees exempts us from having to consider both of these classes simultaneously.  Having said that, I will reserve the right to not draw trivial black vertices when depicting trees graphically.
\end{remark}

\subsubsection{Contraction of white angles}\label{whiteangles}
Finally we point out that contraction of one or more white angles makes sense with marked trees.  We only need to say what happens to the labels of black vertices when two or more are pushed together.  When collapsing a single white angle we use the canonical injection of associahedra $K_n\times K_m \to K_{n+m}$ (given by multiplication) to form the new label, where $n$ corresponds to the cell coming first in the planar order.  Moreover there is a canonical way to contract multiple adjacent white angles of a vertex.  Given a white vertex with $l$ consecutive, nonspined white angles, such that these angles do not comprise the entire vertex, we can simaultaneously collapse these angles.  In so doing we smash together $l+1$ black vertices in to one, which is labeled via the canonical injection $K_l\times K_{n_1}\times\dots\times K_{n_l}\to K_{\sum n_j}$.

\subsection{Grafting branches}\label{grafting} Let $\tau$ and $\sigma$ be marked trees, let $v$ be a vertex of $\tau$, and $u$ a vertex of $\sigma$.  Let $b\in br_v(\tau)$ be a non-root branch and let $f$ be the outgoing flag of $b$.  The flags at $u$ have a cyclic order given by traversing the tree in the planar (clockwise) order.  Let $f_1$ and $f_2$ be consecutive flags at $u$.  Then there is a unique marked tree $\sigma\wedge_{f_1}^{f_2}b$ formed by drawing the branch $b$ as attached to the tree $\sigma$ at $u$ such that $f_1<b<f_2$.  In the case that $u$ is a black vertex labeled by $\mu_n$, the corresponding vertex in $\sigma\wedge_{f_1}^{f_2}b$ is then labeled by $\mu_{n+1}$.  If $u$ is a black vertex with a composite label, blow up the label to a tree and change the $\mu_n$ factor which receives the grafting to $\mu_{n+1}$.  We say that this tree is formed by grafting $b$ to $\sigma$ at $u$.

\section{The Cyclic Deligne Conjecture}\label{sec:cdc}
In this section we will review the statement of the cyclic Deligne conjecture and the solution of Kaufmann in \cite{KCyclic}.

\begin{definition}  A Batalin-Vilkovisky (BV) algebra is an associative, graded-commutative algebra with an operator $\Delta$ of degree $\pm1$ with $\Delta^2=0$, satisfying the BV equation:
\begin{equation}\label{bvequation}
\Delta(abc)=\Delta(ab)c+(-1)^{n}a\Delta(bc)+(-1)^{(n-1)m}b\Delta(ac)-\Delta(a)bc-(-1)^{n}a\Delta(b)c-(-1)^{n+m}ab\Delta(c)
\end{equation}
for $|a|=n,|b|=m,|c|=l$.  
\end{definition}

\begin{lemma} \label{gbvprop}  Let $A$ be a BV algebra. Define a bracket $\{-,-\}$ by
\begin{equation*}
\{ a,b \} :=(-1)^{|a|}\Delta(ab)-(-1)^{|a|} \Delta(a)b-a\Delta(b)
\end{equation*}
Then $\{-,-\}$ is a Gerstenhaber bracket.
\end{lemma}

\begin{theorem}\cite{G94}\label{fldbv}  A $k$-vector space $V$ is a BV algebra if and only if $V$ is an algebra over the homology of the framed little disks operad, $H_\ast(fD_2)$.
\end{theorem} 
The singular chains which correspond under the operad morphism to the operations generating the BV structure, namely the multiplication and the BV operator $\Delta$, are given by a $0$-cell in $S_0(f\op{D}_2)$ and the $1$-cell in $S_1(f\op{D}_2)$ that rotates the outer marked point one complete revolution.

A well known result of \cite{G} is that the Hochschild cohomology of an associative algebra is a Gerstenhaber algebra.  In the cyclic case there is the following extension of this result.

\begin{theorem}\label{bvthm}  Let $A$ be a Frobenius algebra.  Then the Gerstenhaber structure on $HH^\ast(A,A)$ naturally extends to a BV algebra.
\end{theorem}

Full details of this result are given in \cite{Menichi}.  This result is of course still true if $A$ is a cyclic $A_\infty$ algebra; full details can be found in \cite{tradler}.

Now combining Theorem $\ref{bvthm}$ with Theorem $\ref{fldbv}$ we see that, for $A$ Frobenius, $HH^\ast(A,A)$ is an algebra over $H_\ast(fD_2)$, and this fact raises what's known as the cyclic Deligne conjecture.

\begin{theorem}\label{cyclicdeligne}\cite{KCyclic}  There is a chain model for the framed little disks operad that acts on the Hochschild cochains of a Frobenius algebra inducing the standard operations in homology/cohomology.
\end{theorem}
The rest of this section will be devoted to recalling the particulars of Kaufmann's proof that will be needed for generalization beyond the associative case.

\subsection{The BV operator and Normalization}  Again $A$ is a Frobenius algebra and we shall describe the BV operator on $HH^\ast(A,A)$.  Let $B\colon A^{\tensor n+1}\to A^{\tensor n+2}$ be Connes' boundary map from cyclic homology (see e.g. \cite{Loday}).  Explicitly
\begin{equation*}
B(a_0\cdc a_n)=\sum_{i=0}^n (-1)^{ni}(1,a_i\cdc a_n,a_0\cdc a_{i-1})-(-1)^{ni}(a_i,1,a_{i+1}\cdc a_n,a_0\cdc a_{i-1})
\end{equation*}
Now since $A$ is cyclic it is canonically self-dual, and we have a natural isomorphism 
\begin{equation*}
Hom(A^{\tensor n},A)\cong Hom(A^{\tensor n+1},k)=(A^{\tensor n+1})^\ast
\end{equation*}
given by $f \mapsto \langle -,f(-) \rangle$.  Define $\Delta$ as the composition of the following sequence:
\begin{equation*}
Hom(A^{\tensor n+1},A)\cong (A^{\tensor n+2})^\ast\stackrel{B^*}\longrightarrow (A^{\tensor n+1})^\ast\cong Hom(A^{\tensor n},A)
\end{equation*}
Then explicitly if $f\colon A^{\tensor n}\to A$ then 
\begin{equation*}
\langle a_0,\Delta(f)(a_1,\cdc a_{n-1})\rangle = \sum_{i=0}^{n-1}(-1)^{(n-1)i}[\langle 1,f(a_i\cdc a_{i-1})\rangle -\langle a_i,f(1,a_{i+1}\cdc a_{i-1})\rangle]
\end{equation*}
The operator $\Delta$ (defined above on cochains) is compatible with the Hochschild differential and so induces an operator on the Hochschild cohomology, which we also call $\Delta$.  

\subsubsection{The Normalized Hochschild Complex.}
Recall the normalized Hochschild complex; $\overline{CH^\ast}(A,A)\subset CH^\ast(A,A)$ is the subspace consisting of those functions which vanish when evaluated at a pure tensor containing $1$.  Inclusion is a quasi-isomorphism of cochain complexes \cite{Loday}.  As such, if $\Phi$ is an operator on cochains which is compatible with the differential, the induced operator on cohomology depends only on the restriction of $\Phi$ to the normalized cochains.

\begin{remark}\label{normalizationremark}  From now on when we speak of an operation on Hochschild cohomology coming from a cochain level operation we always take the normalized version of that cochain operation as described above.  In other words, in what follows formulas on the cochain level will be written for normalized cochains with the implicit assumption that the operation could be extended to all cochains but that this extension will not effect the cohomology operation.  As an example let us now reconsider the BV operator $\Delta$.  For a normalized cochain the expression for $\Delta(f)$ can be simplified considerably.  Let $t_n =(-1)^{n-1} (1\dots n)$ and $N_n =\sum_{i=0}^{n-1} t^i$.  Then we can write 
\begin{equation}\label{normalbv}
\langle a_0,\Delta(f)(a_1,\cdc a_{n-1})\rangle = \langle 1 , f\circ N_n(a_0,\cdc a_{n-1}) \rangle
\end{equation}
Notice that for normalized cochains we have $\Delta^2=0$ already on the cochain level.
\end{remark}

\subsection{The operad $\op{TS}$}  The chain model used to prove Theorem $\ref{cyclicdeligne}$ takes cells indexed by trees with spines (see Definition $\ref{treeswithspines}$).  Define $\op{TS}(n)$ to be the free $k$-module generated by trees with spines having $n$ white vertices labeled by the numbers $1\cdc n$. 

We will now give the vector space $\op{TS}(n)$ a dg structure.  The degree of a tree with spines will be equal to the number of white vertices whose spine is on a $1$-cell plus the sum of the arity of the white vertices.  The differential is given by taking an alternating sum over all trees which can be found by performing one of the following operations;
\begin{enumerate}\label{anglecontraction}
\item{}  Contraction of a white angle.
\item{}  Take a white vertex whose spine is on a one cell and take the alternating sum of moving the spine to the flag following this $1$-cell in the cyclic order and the flag preceding this $1$-cell in the cyclic order.  This will be called pushing off the spine.
\end{enumerate}
I will denote this differential by $\partial_T$ and remark that $\partial_T$ is also defined on marked trees, since contraction of marked angles makes sense in that context also.  The signs in the differential will be explained below in subsection $\ref{diffsigns}$.

We will now give the collection of dg vector spaces $\op{TS}(n)$ the structure of a dg operad.  Define structure maps
\begin{equation}\label{tsop}
\op{TS}(n)\tensor \op{TS}(m) \stackrel{\circ_i}\longrightarrow \op{TS}(n+m-1)
\end{equation}  
for $i=1,\dots,n$ by 
\begin{equation}\label{tsop2}
\tau_1\circ_i\tau_2 = \sum \pm \tau
\end{equation}
where the sum is taken over all trees with spines that can be formed by the following procedure.  Let $v$ be the vertex of $\tau$ labeled by $i$ and let $\{\tau_1^j\}_j$ be the $v$-branches of $\tau_1$.  This set has a linear order by starting at the spine and going around in the planar order.  Graft these branches on to $\tau_2$, starting by identifying the spine of $\tau_1$ at $v$ with the root of $\tau_2$, in a manner compatible with the cyclic order.  The root of this new tree is the root of $\tau_1$.  Notice that if the root of $\tau_2$ is adjacent to a black vertex then this procedure can create a black edge which we contract as in Definition $\ref{edgecontraction}$.  The signs in equation $\ref{tsop2}$ will be fixed by the choice of orientation of a collection of CW complexes whose cellular chains are isomorphic to $\op{TS}$ as is explained in subsection $\ref{diffsigns}$.

\begin{proposition}\cite{KCyclic}
In this manner, $\op{TS}$ is a $dg$-operad.  In particular this dg operad is isomorphic to the operad of cellular chains on (normalized) cacti.
\end{proposition}
Theorem $\ref{cyclicdeligne}$ can then be proved by exhibiting an action of $\op{TS}$ on the Hochschild cohomology of a Frobenius algebra which is given in \cite{KCyclic}.  For our purposes we would like to `blow up' the operad $\op{TS}$ to something with the additional cells needed to encode operations on a cyclic $A_\infty$ algebra, which is the purpose of the following section.

\section{The Operad $\op{TS}_\infty$}\label{sec:op}  In this section we will define a dg operad of marked trees which we will call   $\op{TS}_\infty$ for `trees with spines with $A_\infty$ labels'.  This operad will be homotopy equivalent to $\op{TS}$ and will eventually serve as our chain model acting on the Hochschild cochains of a cyclic $A_\infty$ algebra.  As a graded $\mathbb{S}$-module we have the following definition.
\begin{definition}  Define $\op{TS}_\infty(n)$ be the graded $k$ vector space generated by marked trees with $n$ white vertices labeled by the numbers $1\cdc n$ (see Definition $\ref{markedtrees}$).  We may denote a marked tree by $(\tau,\{\gamma_i\})$ where $\tau$ is the underlying tree with spines and $\{\gamma_i\}$ is the set of labels of the black vertices $\{v_i\}$ of $\tau$.  For the grading, if $(\tau, \{\gamma_i\})$ is an element of $\op{TS}_\infty$ then we take
\begin{equation*}
|(\tau, \{\gamma_i\})|=|\tau|+\sum_i|\gamma_i|
\end{equation*} 
Finally $\op{TS}(n)$ has an $S_n$ action by permuting the labels of the white vertices. 
\end{definition}

Giving this collection of graded vector spaces the structure of a graded operad will be the subject of this section.  Giving this graded operad the structure of a dg operad is more complicated, and we will postpone the introduction of the differential until subsection $\ref{section:differential}$.

We will now give $\op{TS}_\infty$ the structure of a graded operad.  In the spirit of \cite{KS} and \cite{KCyclic}, $\op{TS}_\infty$ will have an insertion operad structure.  Let $\tau\in\op{TS}_\infty(n)$ and $\sigma\in\op{TS}_\infty(m)$ and let $v$ be the vertex of $\tau$ labeled by $i$ for some $1\leq i \leq n$.  Additionally, let $u$ be the vertex of $\sigma$ which is adjacent to the root.  We define $\tau\circ_i\sigma$ to be a signed sum of trees, i.e.
\begin{equation} \label{signseq2}
\tau\circ_i\sigma=\ds\sum\pm\xi
\end{equation}
The signs in equation $\ref{signseq2}$ will be fixed by the choice of orientation of a collection of CW complexes whose cellular chains are isomorphic to $\op{TS}_\infty$ as will be explained in subsection $\ref{diffsigns}$.  The collection of trees appearing in a sum is determined in three distinct cases as follows.

$\textbf{Case 1:}$  Suppose the vertex $v$ of $\tau$ has a spine on a $1$-cell and also that the vertex $u$ of $\sigma$ is white and has a spine.  In this case we define the collection of trees to be empty and the operadic composition $\tau\circ_i\sigma$ is zero.

$\textbf{Case 2:}$  Suppose $v$ has no spine.  Then the trees appearing in the sum are all those which can be formed by the following procedure.  First detach the $v$-branches of $\tau$, $br_v(\tau)$, identify the root of $\tau$ with the root of $\sigma$, and then graft (see subsection $\ref{grafting}$) the remaining $v$ branches to this configuration such that the cyclic order of the $v$-branches is preserved in the final configuration.  See Figure $\ref{fig:operadstructure}$.

$\textbf{Case 3:}$  Suppose that $v$ has a spine and that $u$ does not.  Write $s$ for the spine of $v$ and notice that the set $\{s\}\cup br_v(\tau)$ has a linear order, starting with the root and traversing $\tau$ clockwise.  Then the trees appearing in the sum are all those which can be formed by the following procedure.  First detach the $v$-branches of $\tau$ and insert the tree $\sigma$ into the vertex $v$ such that the root of $\sigma$ is identified with the spine of $v$.  This (old) root is now the spine of $u$.  Finally graft the $v$-branches to this configuration such that the cyclic order of the set of $\{s\}\cup br_v(\tau)$ is unambiguously preserved in the final configuration.  This means that we do not allow grafting of the root of $\tau$ to $u$ if $u$ is a black vertex with noncomposite label.  The root of $\tau$ is the new root.  See Figure $\ref{fig:operadstructure2}$.

\begin{figure}
	\centering
		\includegraphics[scale=.45]{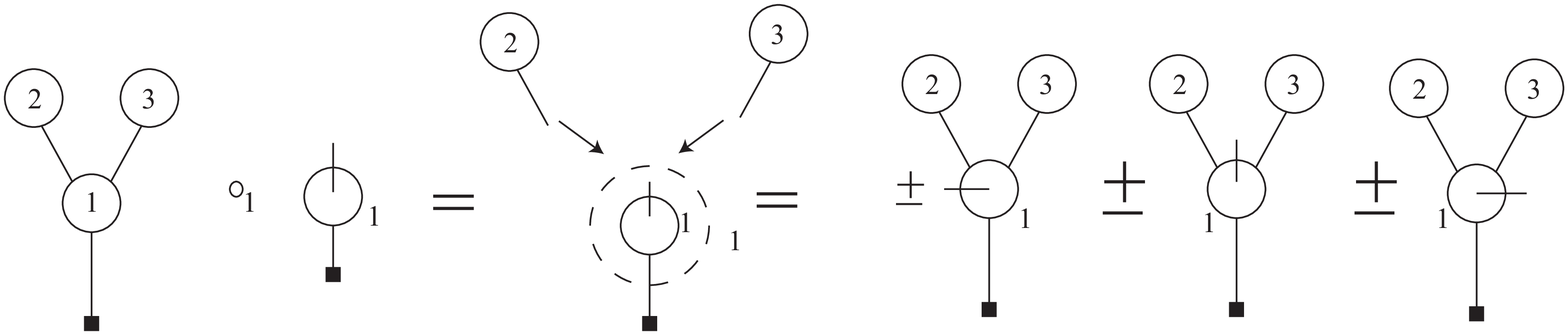}
	\caption{There are three ways to graft the two nonroot branches such that the cyclic order of the branches is preserved. }
	\label{fig:operadstructure}
\end{figure}

\begin{figure}
	\centering
		\includegraphics[scale=.45]{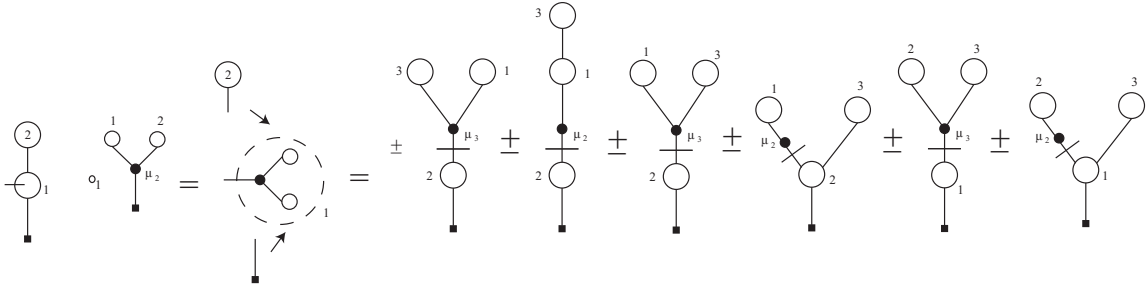}
	\caption{If we first graft the root to vertex $2$ of $\mu_2$ there are four ways to graft the remaining branch.  If we first graft the root to vertex $1$ of $\mu_2$ there are two ways to graft the remaining branch.  If we first graft the root to the black vertex the cyclic order of the set $\{s\}\cup br_v(\tau)$ is ambiguous, and hence this grafting does not contribute.}
	\label{fig:operadstructure2}
\end{figure}

Notice that the key difference in the operad structure of $\op{TS}_\infty$ versus $\op{TS}$ is that now we allow grafting of branches on to black vertices (see subsection $\ref{grafting}$).  In particular, in the former case, the set of terms appearing in the sum in equation $\ref{signseq2}$ includes all of the terms which appear in the latter.  This fact, along with the fact that grafting branches onto black vertices is associative, tells us that the composition defined here is associative and that $\op{TS}_\infty$ is indeed an operad.  

We will make extensive use of the following lemma.
\begin{lemma}\label{simplegluingdef}  Let $\tau$, $\sigma$, and $v$ be as above, and further suppose that $v$ is a leaf with no spine.  Then $\tau\circ_i\sigma$ consists of a single tree, namely the marked tree formed by removing the vertex $v$ and identifying the root of $\sigma$ with the outgoing flag of $\tau$, along with the standard operadic relabeling of white vertices.  See Figure $\ref{fig:cut}$ for an example.  Such a composition will be called a simple composition or a simple gluing.
\end{lemma}

\subsection{Generators of $\op{TS}_\infty$.}  One feature of the operad structure of $\op{TS}_\infty$ is that every marked tree can be decomposed into simple compositions of several classes of generating trees which we will now define.  All generators will have trees with zero or one internal white vertex.  For generators we consider white vertices to be labeled according to the planar order, with the internal white vertex always labeled by $1$ if applicable.  Then the $S_n$ action can produce arbitrary labelings.  The generators are as follows.
\begin{itemize}
\item{}\textbf{Corollas:}  For a spineless corolla with $n$ leaves whose lone black vertex is labeled by $\gamma \in K_n$, we will abuse notation slightly and consider $\gamma \in \op{TS}_\infty(n)$.  The generating corollas are $id \in K_1$ and $\mu_n \in K_n$ for $n\geq 2$.
\item{}\textbf{Delta:}  Let $\delta$ be the unique tree in $\op{TS}_\infty(1)$ with a spine.
\item{}\textbf{Spineless Braces:}  Let $\beta_n$ be the unique spineless tree in $\op{TS}_\infty(n+1)$ with one internal white vertex.
\item{}\textbf{Spined Braces, Type 1:}  Let $\beta_{i,n}$ be the tree formed by taking $\beta_n$ and placing a spine on the internal white vertex between flags $i$ and $i+1$ (mod $n+1$), for $i=0,\dots,n$.
\item{}\textbf{Spined Braces, Type 2:}  For a cell $\gamma_m \in K_m$ and for $1\leq i\leq n$ let $\beta_n\wedge_i\gamma_m$ be the tree formed by taking $\beta_n$ and gluing a corolla with $m-1$ white vertices on to the vertex of $\beta_n$ labeled by $i+1$, and then labeling the new black vertex $v$, having $||v||=m$, by $\gamma_m$ and the white vertices according to the planar order, and placing a spine on the interior white vertex at flag $i$.  For the generators it is actually enough to restrict our labels to $\mu_m$ since the rest can be generated under operadic composition.
\end{itemize}

\begin{figure}[htbp]
	\centering
		\includegraphics[scale=.5]{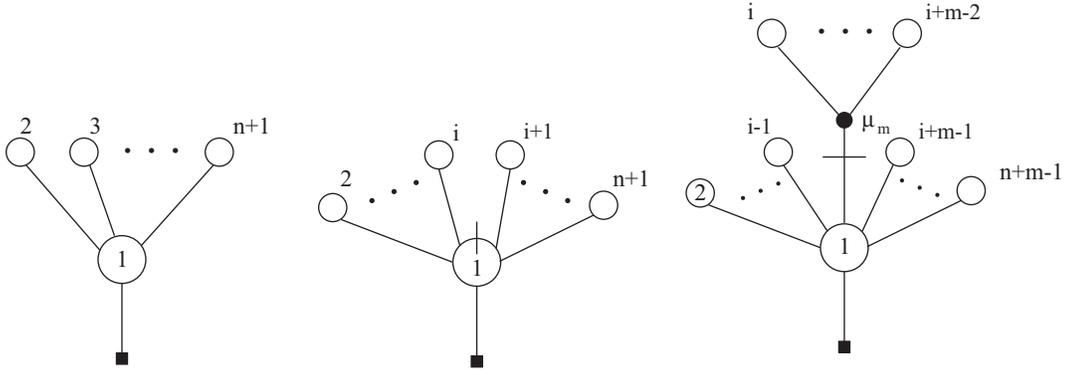}	
	\caption{Left to Right: $\beta_n$, $\beta_{i-1,n}$, and $\beta_{n}\wedge_{i-1}\mu_m$.  The $i-1$ in the subscripts is a result of the fact that the the vertex joining the edge containing the $i^{th}$ flag at $v$ is labeled $i+1$.}
	\label{fig:Bn}	
\end{figure}

\begin{lemma}\label{generators}  Any tree in $\op{TS}_\infty$ can be formed via simple gluings of generators along with the $S_n$ action.
\end{lemma}
\begin{proof}  Let $\tau$ be a marked tree with $n$ white vertices and let $v$ be the white vertex labeled by $i$, and suppose $v$ has $m$ white vertices above it (treewise).  A `cut' of $\tau$ at $v$ is a decomposition of $\tau$ in to two trees, $\tau_1\in \op{TS}(n-m)$ and $\tau_2\in \op{TS}(m+1)$ where $\tau_1$ is given by forgetting everything above $v$, but keeping this white vertex with the spine removed, and $\tau_2$ is given by taking $v$ (and its spine) with everything above it, and adding a root directly below this vertex.  Relabel the white vertices of these trees according to the linear order inherited from the labeling of $\tau$.  Let $j$ be the label of the vertex corresponding to $v$ in $\tau_1$.  Then there is a permutation $\sigma$ such that
\begin{equation}\label{cuteq}
\tau=\sigma(\tau_1\circ_j\tau_2)  
\end{equation}
See Figure $\ref{fig:cut}$ for an example.
\begin{figure}
	\centering
		\includegraphics[scale=.84]{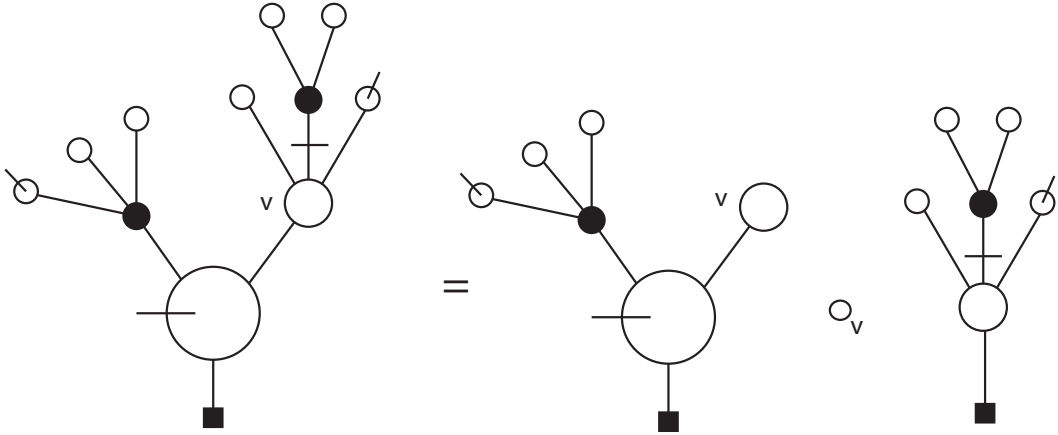}
	\caption{An example of a simple gluing/cut at the vertex $v$.  Notice the composition occurs at leaf $v$ with no spine.  The labels of the other white vertices are immaterial here because they can be freely manipulated via the $S_n$ action, as in equation $\ref{cuteq}$.}
	\label{fig:cut}
\end{figure}

Notice that the vertex of $\tau_1$ which receives the composition is of maximum height and has no spine, as in the statement of the lemma.  As such, a tree can be decomposed by cutting at a given vertex.  Now the height of a tree will be greater than the height of its pieces if we cut at a white vertex of neither maximum or minimum height.  As such we can always decompose a tree (which has finite height) into pieces with no vertices of intermediate height.  In addition we can cut a tree with no white vertices of intermediate height at each white vertex of minimum height to decompose it into pieces with (at most) 1 vertex of minimum height.  Consider first a piece with no internal white vertices.  Such a tree is a corolla whose lone black vertex is labeled by a cell in some $K_m$ and is generated under operadic composition by those corollas labeled by $\mu_n$ for $n \geq 2$.  On the other hand, consider a piece with 1 internal white vertex of minimum height.  Such a tree is of the form corollas glued onto an internal white vertex.  Each flag not labeled with a spine can be cut and replaced with a white vertex.  If a flag is labeled with a spine it can not be cut, and we are left with a type 2 spined brace operation.  We can however blow up the black vertex as indicated by the label and cut it down to a corolla, then relabel by a $\mu_n$.  In other words the type 2 spined brace operations are generated by those whose labeled black vertex takes the label $\mu_n$ for some $n$.  
\end{proof}

\section{The Action}\label{section:action}  In this section we want to give a dg action of $\op{TS}_\infty$ on $CH^\ast(A,A)$, where $A$ is a cyclic $A_\infty$ algebra.  I will call the morphism of operads $\rho$, i.e.
\begin{equation*}
\op{TS}_\infty\stackrel{\rho}\to End_{CH^\ast(A,A)}
\end{equation*}
We will first define the action informally via a generalization of the foliage operator of \cite{KCyclic}.  We will then take the time to give a precise description of the action $\rho$ starting with the generators given in Lemma $\ref{generators}$ and then extending to all trees in $\op{TS}_\infty$.  The formulas written below are for normalized cochains in accordance with Remark $\ref{normalizationremark}$.

\subsection{The Foliage Operator.}  Let $\tau \in \op{TS}_\infty(n)$ be a generator.  Let $F(\tau)$ be the formal sum over all ways to attach free tails to the white vertices of $\tau$ $\textit{and}$ the black vertices of $\tau$ which are labeled with a cell of $K_n$ for $n\geq 2$, changing the labels accordingly.  Notice that we do not allow the attaching of free tails to trivial black vertices.  A marked tree with free tails attached will be called a tree with foliage and $F$ will be called the foliage operator.  We can now informally describe the action on generators.  For a tree with foliage $\tau^\prime$ we define $\rho(\tau^\prime)(f_1\cdc f_n)$ to be zero unless for each $i$ the number of free tails of vertex $i$ is $|f_i|$ in which case $\rho(\tau^\prime)(f_1\cdc f_n)$ is a signed multiple of an operation on $CH^n(A,A)$ described informally as follows.  Insert $f_i$ into the vertex labeled by $i$, label the spine by $1$, the root by $a_0$ and each of the free tails by $a_i$ in the cyclic order, and then read off the result as a flow chart starting at the spine and moving clockwise to form an operation $\Phi$.  Finally, the image of $\tau^\prime$ is defined implicitly as
\begin{equation*}
\langle 1,\Phi(f_1,\dots,f_n)(a_0\cdc a_N) \rangle = \langle a_0, \rho(\tau^\prime)(f_1,\dots,f_n)(a_1\cdc a_N) \rangle
\end{equation*}
Notice that in the case of a spineless generator the root and spine coincide, and the image of $\rho$ can be described without the bilinear form.  We can then define the action for the generator $\tau \in \op{TS}_\infty$ by
\begin{equation*}
\rho(\tau):=\rho(F(\tau))
\end{equation*}
 
The sign associated to each tree with foliage $\tau^\prime$ can be determined by giving the spine, the root, the free tails, and the incoming edges adjacent to trivial vertices weight one.  The sign then is the product over all vertices of the sign of the permutation which permutes the planar order of the weighted elements at each vertex to the order having the root (if applicable), the edges, the free tails, and then the spine (if applicable).

We have now given an informal description of the action and the signs involved.  The remainder of this section will be dedicated to a precise description of the action on generators and an extension of this action to the whole of $\op{TS}_\infty$.  We simply observe here that the precise description given below (which we take as the definition of $\rho$) agrees with the informal description given by the Foliage operator (which we take as intuition).

\subsection{Distinguished cochains} In this subsection we will describe the operations on cochains which are in the image of the generators of $\op{TS}_\infty$.  These include the brace operations of Getzler/Kadeishvili \cite{Brace1}, \cite{Brace2} and a generalization of these operations which I will call spined braces which we now define.

\begin{definition}  The brace operations $B_n \in End_{CH^\ast(A,A)}(n+1)$ are defined by
\begin{equation}\label{braceeq}
B_n(f;g_1,\dots g_n):=  \ds\sum_{1\leq i_1 <\dots<i_n \leq m}\pm (\dots(f\circ_{i_n}g_n)\circ_{i_{n-1}} g_{n-1})\dots \circ_{i_1} g_1)
\end{equation}
where $|f|=m$ (See Notation $\ref{degreenotation}$).  
\end{definition}
Notice
\begin{equation*}
|B_n(f;g_1,\dots g_n)|=m-n+\ds\sum_{i=1}^n |g_i|
\end{equation*}
The signs in the case that $A$ is ungraded are given in e.g. \cite{Brace1}.  In the case where $A$ is graded, e.g. if $A$ is a cyclic $A_\infty$ algebra, there are additional signs as in equation $\ref{preliesigns}$.  In each term of the sum in the equation $\ref{braceeq}$ the sign can be determined by iterating the sign appearing in equation $\ref{preliesigns}$.  In particular $B_1(f;g)=f\circ g$.  Typical notation for the brace operation $B_n$ evaluated at functions $f,g_1,\dots,g_n$ is
\begin{equation*}
B_n(f;g_1,\dots,g_n)=f\{g_1,\dots,g_n\}
\end{equation*}
I will use both notations, since we will often have the need to specify the brace operation itself, not just when it's evaluated at functions.  We will also need what I will call spined brace operations.  There are two types. 

\begin{definition}  (\textbf{Type 1.})   We define $B_{l,n} \in End_{CH^\ast(A,A)}(n+1)$ for $l=0,\dots, n$ implicitly by 
\begin{equation*}
\langle a_0,B_{l,n}(f;g_1\cdc g_n)(\alpha)\rangle = \langle 1, B^1_{l,n}(f;g_1\cdc g_n) (a_0,\alpha) \rangle
\end{equation*}
where $|f|=m$, $\alpha \in A^{\tensor (m-n+\sum_i |g_i|)}$ and where
\begin{equation*}
B^1_{l,n}(f;g_1,\dots g_n)(a_0,-):=  \ds\sum_{\substack{1\leq i_{l+1} <\dots<i_n\\ <i_0<\dots < i_l \leq m}}\pm (\dots(((((f\circ_{i_l}g_l)\circ_{i_{l-1}} g_{l-1})\dots \circ_{i_1} g_1)\circ_{i_0}a_0)\circ_{i_n}g_n) \circ \dots g_{l+1})\circ t^j
\end{equation*}
where $t$ is the permutation $(1\dots |\alpha|)$ and where
\begin{equation*}
j= i_0+\ds\sum_{r=l+1}^n (|g_r|-1)
\end{equation*}  The signs are determined by including the signs in equation $\ref{preliesigns}$ in each composition along with an additional $(-1)^{j(|\alpha|-1)}$ to account for the $t^j$.  The intuition behind this definition comes from considering the tree $\beta_{l,n}$, where the spine is labeled by 1, the root is labeled by $a_0$, the internal vertex is labeled by $f$, the leaves are labeled by the $g_i$ and the arguments $a_i$ are freely adjoined to the vertices as per the cyclic order.  Notice that $a_0$ plays a special role in that it must be between $g_n$ and $g_1$.  For $B^1_{l,n}$ evaluated at functions which are in turn evaluated at a pure tensor we will also use the following imprecise but intuitive notation;
\begin{equation*}
B^1_{l,n}(f,g_1\cdc g_n)=f\{'g_{l+1}\cdc g_n, a_0,g_1\cdc g_l \}
\end{equation*}
where the element $a_0\in A$ is the left most term of the pure tensor.  
\end{definition}

\begin{definition}  (\textbf{Type 2.})   We define $B_n\wedge_l\gamma_m \in End_{CH^\ast(A,A)}(n+m-1)$ for $\gamma_m \in K_m$ and $l=1,\dots, n$ implicitly by
\begin{equation*}
\langle a_0,B_n\wedge_l\gamma_m(f;g_1,\dots g_{n+m-1})(\alpha)\rangle=\langle 1,B^1_n\wedge_l\gamma_m(f;g_1,\dots g_{n+m-1})(a_0,\alpha)\rangle
\end{equation*}
where $\alpha \in \mathbb{T}A$, the tensor algebra, is a direct summand of pure tensors and where 
\begin{equation}\label{defeq}
B^1_n\wedge_l\gamma_m(f;g_1,\dots g_{n+m-1})(a_0,\alpha)=F(\gamma_m)(g_l\cdc g_{l+m-2},\psi)
\end{equation}
Where $\psi=$
\begin{equation*}
\ds\sum_{\substack{1\leq i_{l+m-1} <\dots<i_{n+m-1}\\ <i_0<\dots < i_{l-1} \leq m}}\pm (\dots((((f\circ_{i_{l-1}}g_{l-1})\dots \circ_{i_1} g_1)\circ_{i_0}a_0)\circ_{i_{n+m-1}}g_{n+m-1}) \circ_{i_{l+m-1}} \dots g_{l+m-1})\circ t^j
\end{equation*}
Here by a function evaluated at $\alpha \in \mathbb{T}A$ we mean evaluate the function at the summand $\alpha_s$ of $\alpha$ having the correct number of inputs (which changes with the foliage) and $t$ denotes the permutation $(1\dots s)$ (again $s$ changes with the foliage) and where $j$ is chosen such that the first argument of the relevant summand of $\alpha$ immediately follows $a_0$.  Again the signs in each term of the sum are determined by including the signs in equation $\ref{preliesigns}$ at each composition as well as $(-1)^{j(|\alpha_s|-1)}$.  We will write 
\begin{equation*}
B^1_n\wedge_l\gamma_m(f;g_1,\dots g_{n+m-1})=F(\gamma_m)(g_l\cdc g_{l+m-2},f\{'g_{l+m-1}\cdc g_{n+m-1},a_0,g_1 \cdc g_{l-1}\})
\end{equation*}
where the element $a_0$ will be clear from the context.
\end{definition}
\begin{remark}  In the previous definition we assume that the cell $\gamma_m \in K_m$ is labeled according to the planar order (as was our assumption for all generators).  For other non-generating trees, this need not be the case and the order of the functions in equation $\ref{defeq}$ will reflect the labeling of the black vertex.
\end{remark}

\subsection{Action of the Generators.}\label{actionsection}
We now spell out the action of the generators:
\begin{itemize}
\item  \textbf{Corollas:}  Given a corolla whose lone black vertex is labeled by $\gamma \in K_n$ for $n \geq 2$ we can blow up this vertex to the corresponding tree with white leaves and all other vertices black.  The corolla is then mapped under $\rho$ to all ways of attaching free tails (elements) to this picture and then multiplying as specified.  In particular if we restrict our attention to generators then we can take $\gamma=\mu_n$ and then $\rho(\mu_n)=F(\mu_n)\in End_{CH^\ast(A,A)}(n)$ where
\begin{equation*} 
F(\mu_n)(f_1\cdc f_n)(\alpha)=\bigoplus_{s\geq n}\mu_s\{f_1\cdc f_n\}(\alpha_k)
\end{equation*}
where $\alpha \in \mathbb{T}A$ and $k=s-n+\sum_i|f_i|$.  Notice that the corolla with one white vertex is mapped to the identity operation.
\item  \textbf{The BV operator:}  The tree $\delta$ is mapped to the BV operator, defined implicitly using the bilinear form;
\begin{equation*}
\langle a_0, \Delta(f)(a_1,\dots, a_{n-1})\rangle = \langle 1, f\circ N(a_0,\dots, a_{n-1})\rangle
\end{equation*}
for $f \in CH^n(A,A)$, as we have seen above (Remark $\ref{normalizationremark}$).
\item  \textbf{Unspined Braces:} We define $\rho(\beta_n)=B_n$.  Using the bracket notation for brace operations we write
\begin{equation*}
\rho(\beta_n)(f,g_1,\dots, g_n)=f\{g_1,\dots, g_n\}
\end{equation*}
\item \textbf{Spined Braces, Type 1:}  We define $\rho(\beta_{i,n})=B_{i,n}$.  Recall that this defines $\rho(\beta_{i,n})$ implicitly by 
\begin{align*}
 & \langle a_0, \rho(\beta_{i,n})(f,g_1\cdc g_n)(\alpha)\rangle=\langle a_0, B_{i,n}(f;g_1,\dots, g_n)(\alpha)\rangle \\
 & \indent = \langle 1, B_{i,n}^1(f;g_1,\dots, g_n)(a_0,\alpha)\rangle = \langle 1, f\{'g_{i+1}\cdc g_n, a_0, g_1 \cdc g_i \}(\alpha)\rangle
\end{align*}
\item \textbf{Spined Braces, Type 2:}  We define $\rho(\beta_n\wedge_i \mu_m)=B_n\wedge_i\mu_m$.  Then
\begin{align*}
& \langle a_0, \rho(\beta_n\wedge_i \mu_m)(f,g_1\cdc g_{n+m-1})(\alpha)\rangle=\langle a_0, B_n\wedge_i\mu_m(f;g_1,\dots, g_{n+m-1})(\alpha)\rangle\\
& \indent = \langle 1, B^1_n\wedge_i\mu_m(f;g_1,\dots, g_{n+m-1})(a_0,\alpha)\rangle \\
& \indent = \langle 1, F(\mu_m)(g_i\cdc g_{i+m-2},f\{'g_{i+m-1}\cdc g_{n+m-1},a_0,g_1 \cdc g_{i-1}\})(\alpha) \rangle 
\end{align*}
But due to the invariance of the bilinear form and the fact that $\mu_n$ is normalized for $n\geq 3$ (Remark $\ref{unitremark}$), this is zero unless we have $m=2$, in which case all foliage is zero, but the original terms remain.  That is $\langle a_0, \rho(\beta_n\wedge_i \mu_m)(f,g_1\cdc g_{n+m-1})(\alpha)\rangle =$
\begin{equation}\label{zerocases}
\begin{cases} \langle 1, g_if\{'g_{i+1}\cdc g_{n+m-1},a_0,g_1 \cdc g_{i-1}\}(\alpha) \rangle  & \text{ if } m=2 \\ 0 & \text{else}\end{cases}
\end{equation}
\end{itemize}
I will remark here that the action on spineless generators coincides with that of \cite{KS} and \cite{KSch} and the action on spined brace operations coincides with that of \cite{KCyclic}.  In order to extend $\rho$ to a morphism on the entirety of $\op{TS}_\infty$ we proceed as follows.  Let $\tau$ be any tree appearing in $\op{TS}_\infty$ and choose a decomposition into generators, 
\begin{equation*}
\tau= (...(g_1\circ_{i_1}\dots\circ_{i_{n-1}}g_n)\dots)
\end{equation*}
such that composition is simple (see Lemma $\ref{simplegluingdef}$).  There is at least one such decomposition by Lemma $\ref{generators}$.  Here the notation $(\dots($ means that there is a parenthisization of these generators which we do not wish to specify.  We then define
\begin{equation*}
\rho(\tau)= (...(\rho(g_1)\circ_{i_1}\dots\circ_{i_{n-1}}\rho(g_n))\dots)
\end{equation*}
where the parethesization is the same as in the original presentation of $\tau$.
\begin{lemma}  The definition of $\rho(\tau)$ is well defined independent of the choice of such a decomposition.
\end{lemma}
\begin{proof}  Since each composition occurs at a vertex of maximum height with no spine, such a decomposition corresponds to a decomposition of a tree (forgetting the extra data), and so the operad associativity of $End_{CH^\ast(A,A)}$ ensures the lemma.
\end{proof}
We then extend $\rho$ linearly across linear combinations of trees.
\begin{lemma}\label{morphismlemma}  If $\tau$ and $\tau^\prime$ are generators of $\op{TS}_\infty$ then $\rho(\tau\circ_i \tau^\prime) = \rho(\tau)\circ_i\rho(\tau^\prime)$.
\end{lemma}
\begin{proof}
In order to show that $\rho$ respects the composition on generators we can limit our horizons as follows.  We will see below that this action coincides with that of \cite{KS} on the suboperad generated by spineless trees, thus we do not need to consider spineless operations.  Additionally, since the brace operations and the spined brace operations coincide with those of \cite{KCyclic} we do not need to consider these compositions.  Finally, since composing with a corolla on the left is a composition at a spineless vertex of maximum height, this case follows from the definition.  Thus the only compositions of generators that remain to be checked are the following:
\begin{eqnarray*}
\rho(\beta_{m,l} \circ_1 \mu_n) &=& \rho(\beta_{m,l})\circ_1 \rho(\mu_n) \label{eqI}\\ 
\rho((\beta_m\wedge_l\mu_p) \circ_1 \mu_n) &=& \rho(\beta_m\wedge_l\mu_p) \circ_1 \rho(\mu_n) \label{eqII}
\end{eqnarray*}
for $n\geq 3$.  So let us prove equation $\ref{eqI}$.  For $f_1\cdc f_n,g_1\cdc g_m$ cochains we have
\begin{eqnarray*}
\langle a_0, \rho(\beta_{m,l})\circ_1\rho(\mu_n)(f_1\cdc f_n,g_1\cdc g_m) \rangle 
= \langle a_0, \rho(\beta_{m,l})(F(\mu_n(f_1\cdc f_n)),g_1\cdc g_m) \rangle \\
= \langle a_0, B_{n,l}[F(\mu_n(f_1\cdc f_n));g_1\cdc g_m] \rangle 
= \langle 1, [F(\mu_n(f_1\cdc f_n))]\{g_{l+1}\cdc g_m,a_0,g_1\cdc g_l\}\rangle\\
\end{eqnarray*}
There are two ways in which a summand of this expression could be nonzero.  One is if there is no foliage, in which case each $g_i$ must be inserted in to some $f_j$, which corresponds to grafting branches to white vertices in the composition $\beta_{m,l} \circ_1 \mu_n$.  The other is if there is foliage that is filled with some of the $g_i$.  Such terms correspond to grafting branches to the black vertex in $\beta_{m,l} \circ_1 \mu_n$ (see Figure $\ref{fig:operadstructure2}$ for an example of such a composition).  Any terms with foliage not filled with the $g_i$ will be zero since we can rotate the $1$ in to this tail using the invariance of the bilinear form, and then appeal to our normalized assumption.
\end{proof}

\begin{theorem}\label{opthm}  Let $A$ be a cyclic $A_\infty$ algebra.  The map $\rho\colon\op{TS}_\infty\to End_{CH^\ast(A,A)}$ defines a morphism of operads.
\end{theorem}
\begin{proof}  Since $\rho$ is defined to be linear over trees it is enough to show that $\rho$ respects the operadic composition of two trees.  That is, if $\tau$ and $\tau^\prime$ are trees in $\op{TS}_\infty$ we must show
\begin{equation}\label{proofeq}
\rho(\tau\circ_i \tau^\prime) = \rho(\tau)\circ_i\rho(\tau^\prime)
\end{equation}
We will show this using the following two facts which follow from the Definition of $\rho$ and from Lemma $\ref{morphismlemma}$:
\begin{enumerate}
\item  Equation $\ref{proofeq}$ holds if the gluing is simple.  
\item  Equation $\ref{proofeq}$ holds if $\tau$ and $\tau^\prime$ are generators.
\end{enumerate}
Given a tree $\tau$ we can write $\tau$ as a composition of generators $\tau = (\dots(h_1\circ_{i_1}\dots\circ_{i_{m-1}} h_m)\dots)$ such that each composition occurs at a vertex (of the tree on the left hand side of $\circ$) of maximum height with no spine.  Again the notation $(\dots($ means that there is some parenthesization of these binary operations, which we don't what to specify.  Then property $(1)$ above tells us that $\rho(\tau)=(\dots(\rho(h_1)\circ_{i_1}\dots\circ_{i_{m-1}} \rho(h_m))\dots)$.  Now suppose that $\tau^\prime$ can be written as a composition of $n$ generators $g_1\cdc g_n$, and induct on $n$.  If $n=1$ then write $\tau$ as a composition of generators $h_1\cdc h_m$ such that each composition occurs at a vertex of maximum height.  Let $h_j$ be the generator appearing in $\tau$ which contains the vertex $v$.  Then using the associativity axiom we can write
\begin{equation*}
\tau\circ_i g_1 = (\dots(\dots\circ(h_j\circ_vg_1)\circ \dots)\dots)
\end{equation*}
Where each binary composition in the above expression is either a composition at maximum height or a composition of generators.  Thus by properties (1) and (2) above and the associativity axiom we can apply $\rho$ and see that 
\begin{eqnarray*}
\rho(\dots(\dots\circ(h_j\circ_vg_1)\circ \dots)\dots)= (\dots(\dots\circ(\rho(h_j)\circ_v\rho(g_1))\circ \dots)\dots) \\ 
= (\dots(\rho(h_1)\circ_{i_1}\dots\circ_{i_{m-1}} \rho(h_m))\dots)\circ_v \rho(g_1) = \rho(\tau)\circ_i \rho(g_1)
\end{eqnarray*}
For the induction step notice that we can use the associativity axiom to move the outer parentheses of $\tau\circ_v\tau^\prime=\tau \circ_v (\dots(g_1\circ_{l_1}\dots \circ_{l_{n-1}} g_n)\dots)$ and invoke the induction hypothesis.
\end{proof}

\begin{remark}\label{gradingconventions2}
Note that $\rho$ actually reverses the grading; for example $|B_n|=-n$ and $|\beta_n|=n$, additionally $|\delta|=1$ and $|\Delta|=-1$.  Hence $\rho$ induces a morphism of graded operads on the operad with grading opposite of $\op{TS_\infty}$.  Keeping this in mind we will consider $\rho$ a morphism of graded operads.  The reason why this reversal of grading is necessary is because on the one hand there is a standard grading on the Hochschild cochains which takes the number of inputs as $\textit{adding}$ to the grading, and on the other hand we will build a CW complex whose cells are indexed by trees, where an edge of a tree corresponds to $\textit{taking away}$ an input.  We could circumvent this issue by choosing nonstandard grading conventions for the Hochschild complex (Remark $\ref{gradingconventions}$).
\end{remark}

\subsection{The Differential.}\label{section:differential}  In the previous subsection we constructed a morphism of graded operads
\begin{equation*}
\op{TS}_\infty\stackrel{\rho}\to End_{CH^\ast(A,A)}
\end{equation*}
We would now like to define a differential $\partial_H$ such that 
\begin{enumerate}
\item{}  $\rho$ is a morphism of dg operads, i.e. $\rho \circ \partial_H = d_{Int}\circ \rho$.
\item{}  The operads $(\op{TS}_\infty, \partial_H)$ and $(\op{TS}, \partial_T)$ are homotopy equivalent. 
\end{enumerate}
Here we use the notation $d_{Int}$ for the differential on $End_{CH^\ast(A,A)}$ and call this the `internal differential'.  Explicitly for $\Phi \in End_{CH^\ast(A,A)}(n)$, 
\begin{equation}\label{internaleq}
d_{Int}(\Phi)(f_1,\dots,f_n):= \ds\sum_{i=1}^n (-1)^{||f_1||+\dots+||f_{i-1}||}\Phi(f_1\cdc d(f_i)\cdc f_n)-(-1)^{|\Phi|} d(\Phi(f_1\cdc f_n))
\end{equation} 
where $d$ on the right hand side is the Hochschild differential $d(f)=[f,\mu]$.  The terminology `internal' is meant to remind one that $d_{Int}$ takes $End_{CH^\ast(A,A)}(n)$ to itself, as opposed to the Hochschild differential which changes the number of inputs.

For a marked tree $\tau$ we define the differential $\partial_H$ `locally'.  That is we first define the differential at a vertex $v$, call this $\partial(\tau;v)$, and then define
\begin{equation*}
\partial_H(\tau):=\sum_{v\in \tau}\pm \partial(\tau;v)
\end{equation*}
To define $\partial(\tau;v)$ we have three cases.

$\textbf{Case 1}$: $v$ is a black vertex.  In this case we define $\partial(\tau;v)$ to be the tree resulting from taking the associahedra differential on the label of $v$.  Notice then that trivial black vertices do not contribute to the differential. 

$\textbf{Case 2}$: $v$ is a white vertex whose spine is on a $0$ cell.  In this case $\partial(\tau;v)$ is a sum of all trees which can be formed by contracting one or more adjacent white angles (see subsection $\ref{whiteangles}$).

$\textbf{Case 3}$: $v$ is a white vertex whose spine is on a $1$ cell.  In this case $\partial(\tau;v)$ is a sum of all trees which can be formed by contracting one or more adjacent spineless white angles, along with the two trees that can be formed by pushing off the spine. 
\begin{figure}[htbp]
	\centering
		\includegraphics[scale=.5]{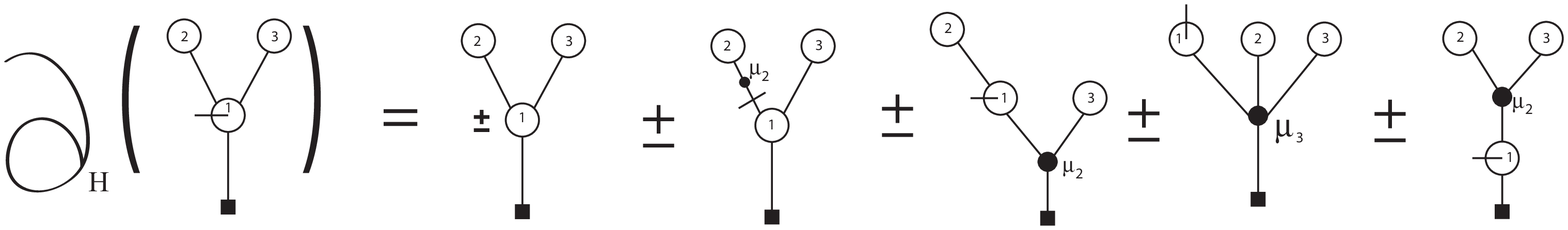}
	\caption{ $\partial_H(\beta_{0,2})$ is the signed sum of 5 trees.  In the picture the first two trees come from pushing off the spine, the third and fifth come from collapsing a single white angle and the fourth comes from collapsing the two consecutive nonspined white angles.  The difference between the differentials $\partial_T(\beta_{0,2})$ and $\partial_H(\beta_{0,2})$ is the existence of the fourth tree (with $\mu_3$ as a label) in the latter.}
	\label{fig:boundary}
\end{figure}

\begin{figure}[htbp]
	\centering
		\includegraphics[scale=.45]{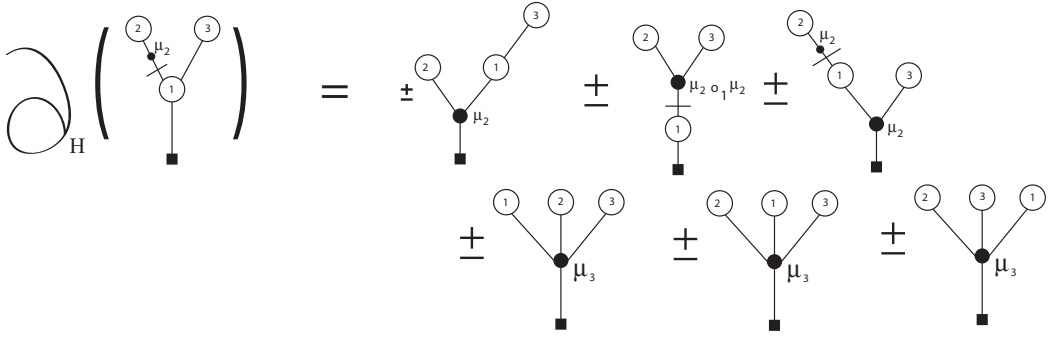}
	\caption{ $\partial_H(\beta_2\wedge_1\mu_2)$ is the signed sum of 6 trees.  The 3 trees on the top line come from collapsing 1 white angle, and the 3 trees on the bottom line come from collapsing 2 consecutive white angles.}
	\label{fig:boundary2}
\end{figure}

For examples see Figures $\ref{fig:boundary}$ and $\ref{fig:boundary2}$.  The signs will be fixed by the choice of orientation of the cells of a certain collection of CW complexes as we will explain below (Subsection $\ref{diffsigns}$).  The fact that $\op{TS}_\infty(n)$ is isomorphic to the cells of a CW complex also tells us that $\partial_H^2=0$.  It follows from this definition that $\partial_H(\mu_2)=\partial_H(\delta)=\partial_H(\beta_1-(12)\beta_1)=0$.  These three cycles will represent the product, the BV operator, and the Gerstenhaber bracket in homology.

The local description of the differential $\partial_H$ given above facilitates checking that the operadic composition maps are dg, i.e. that $\partial_H(\tau\circ_i\sigma)=\partial_H(\tau)\circ_i\sigma+(-1)^{|\tau|}\tau\circ_i\partial_H(\sigma)$.  To see this let $v$ be the vertex of $\tau$ labeled by $i$ and pick a vertex $u\neq v$ of $\tau$ and $w$ of $\sigma$.  Then clearly $\partial(\tau\circ_i\sigma;u)=\partial(\tau;u)\circ_i\sigma$, and the terms appearing in $\partial(\tau;w)\circ_i\sigma$ also appear in $\partial(\tau\circ_i\sigma;w)$, by grafting no branches into the relevant white angle in the latter.  Moreover the terms appearing in $\partial(\tau;v)\circ_i\sigma$ correspond to terms appearing in $\partial_H(\tau\circ_i\sigma)$ which graft multiple branches into a given white angle of $\sigma$ and then collapse the newly created white angles.  Finally note that the terms appearing in $\partial(\tau\circ_i\sigma;u)$ which collapse a newly created white angle between a grafted and nongrafted branch can be created in two ways with opposite sign.

Define $\op{T}_\infty(n) \subset \op{TS}_\infty(n)$ to be the vector space generated by spineless marked trees.  Notice the collection $\op{T}_\infty = \{\op{T}_\infty(n)\}$ forms a suboperad of $\op{TS}_\infty$ generated by the corollas and brace operations.  We defined the differential $\partial_H$ on these generators above and thus can extend $\partial_H$ to all of $\op{T}_\infty$. 

\begin{lemma}\label{minlemma}  Let $M$ be the minimal operad of \cite{KS}.  As dg operads $(\op{T}_\infty,\partial_H) \cong M$.
\end{lemma}
\begin{proof}  By \cite{KSch} the operad $M$ is equivalent to a dg operad indexed by so called stable trees which are b/w planar planted trees with no black vertices of arity one.  Stable trees and marked trees are seen to be equivalent by the above work.  In particular given a stable tree we can collapse all black edges (labeling as we go) and insert free black vertices in to white edges to get a marked tree.  This operation is also invertible.  Then it is straight forward to check that the operad structure and the differential are the same.
\end{proof}

\begin{theorem}\label{spinelessdiff}\cite{KS}  The differential $\partial_H$ and the above action constitute a $dg$ action of $M\cong\op{T}_\infty$ on the endomorphism operad of the Hochschild cochains of an $A_\infty$ algebra.
\end{theorem}

Our description of the minimal operad $(\op{T}_\infty, \partial_H)$ and its action on the Hochschild complex reduces a proof of this statement to checking that the action is compatible with the differentials for the generators $\mu_n$ and $\beta_n$.  This can be verified by a direct but lengthy calculation.

Theorem $\ref{spinelessdiff}$ tells us that the differential $\partial_H$ is compatible with the action $\rho$ on the suboperad $\op{T}_\infty$.  This is true on the entirety of $\op{TS}_\infty$, as we now record.
\begin{theorem}\label{diffthm}  $\partial_H$ is compatible with $\rho$.
\end{theorem}
The remainder of the proof of this statement will be relegated to the appendix.  Let us note, however that it is enough to check the compatibility on generators, since we can write $\partial_H(\tau)$ as
\begin{equation}\label{diffext}
\partial_H[(\dots(g_1\circ_{i_1}\dots \circ_{i_n}g_{n+1})\dots)] =\sum_{j=1}^n(-1)^{|g_1|+\dots+|g_{j-1}|}(\dots(g_1\circ_{i_1}\dots \circ_{i_{j-1}}\partial_H(g_j)\circ_{i_j}\dots\circ_{i_n}g_{n+1})\dots)
\end{equation}
for generators $g_j$, with the corresponding equation also being true on the Hochschild side.

\section{Underlying Topology}\label{section:underlyingtopology}  The purpose of this section will be to prove that $\op{TS}_\infty$ is a chain model for $fD_2$.  In order to do this we will consider a collection of CW complexes $\op{X}=\{X_n\}$ which are homotopy equivalent to the framed little disks and whose cellular chains are $\op{TS}_\infty$.  Before considering the underlying topology of the nonassociative case $\op{TS}_\infty$ we will review the associative case $\op{TS}.$

\subsection{The associative case: $CC_\ast(Cacti^1)$}  In the associative case the cell model $\op{TS}$ can be realized as the cellular chains of the topological (quasi)-operad of normalized cacti, as we now review.
\subsubsection{$\op{C}acti$}  Let us briefly recall the topological operad $\op{C}acti$ (see e.g. \cite{Vor},\cite{CV},\cite{KCacti} for more detail).  The space $\op{C}acti(n)$ consists of the collection of all $\{1\cdc n\}$ labeled treelike configurations of $n$ parameterized circles each with a specified perimeter, called lobes, along with a marked point on each lobe, called a spine, a cyclic order at each intersection of lobes, and a marked point associated to the entire configuration, called the global zero.  We consider such configurations as drawn in the plane with lobes oriented counterclockwise.  In this way the arcs of the lobes are labeled with their length.  The operad structure map $\circ_i$ is given by inserting a configuration into the $i^{th}$ lobe of another configuration by identifying the global zero of the former with the spine of the $i^{th}$ lobe of the latter.  In accordance with \cite{KCacti} we call a cacti normalized if the perimeter of each lobe is $1$.  The advantage of normalized cacti is the presence of a natural CW structure given by considering the lengths of the arcs (which sum to 1) of a lobe as corresponding to the points in a simplex of the appropriate degree and taking the product over all lobes.  The disadvantage of normalized cacti is that they do not form a topological operad because their composition does not preserve their defining property.  However a cactus can be normalized or rescaled, and in this way normalized cacti form what is called a topological quasi-operad in \cite{KCacti}.  Moreover this structure induces an honest operad structure on the cellular chains.  In particular we have,

\begin{proposition}\cite{KCacti},\cite{KCyclic}  The cellular chains of normalized cacti form a dg operad and there is an isomorphism of dg operads
\begin{equation*}
CC_\ast(Cacti^1)\cong \op{TS}
\end{equation*}
In particular $\op{TS}$ can be thought of as the cellular chains of a collection of CW complexes.
\end{proposition}
Let $|\op{TS}|$ denote the CW complexes whose cellular chains are $\op{TS}$.  We will consider $|\op{TS}|\cong \op{C}acti^1$, instead of considering them to be precisely equal.  To understand the space $|\op{TS}(n)|$ it should be considered as a CW complex whose $m$-cells are indexed by trees with spines of degree $m$ with $n$ white vertices.  The points in this space correspond to `weighted trees':  trees whose interior white vertices are labeled by weights at each $1$ cell which sum to $1$ at each vertex.  The weight at an arc segment of $S^1$ is meant to correspond to the length of the arc.  In particular $|\overline{\delta}|\cong S^1$.  Additionally, since $\beta_n$ has $1$ interior white vertex, no spines, and no nontrivial black vertices, $|\overline{\beta_n}|$ is homeomorphic to an $n$-simplex, and this homeomorphism induces an isomorphism of cellular chains (taking the standard cellular structure of an $n$-simplex).  Here we use the notation $\overline{\tau}$ for the complex generated by a tree $\tau$ and the differential.
\subsubsection{Compatibility of the BV Operator and the Brace Operations} 
Let $\Delta_n$ denote the standard $n$-simplex with vertices $e_0\cdc e_n$.  For $i=1\cdc n$ define a simplicial map $D_i\colon \Delta_n \to \Delta_{n-1}\times I$ by
\begin{equation*}
D_i(e_j)= \begin{cases} (e_j,0) & \text{ if } j\leq i \\ (e_{j-1},1) & \text{ if } j>i \end{cases}
\end{equation*}
Then $D_i(\Delta_n)\cap D_{i+1}(\Delta_n)\cong \Delta_{n-1}$ and the images of the $D_i$ decompose $\Delta_{n-1}\times I$ into $n$ copies of $\Delta_n$.  This decomposition induces a natural cellular structure on the space $\Delta_{n-1}\times I$ that is finer than the standard product structure.  We hereafter consider $\Delta_n\times I$ to be a CW complex with cells given by this finer decomposition.  Quotienting $\Delta_n\times I$ by the equivalence relation $(x,0)\sim(x,1)$ gives the induced cellular structure to $\Delta_n\times S^1$.  This decomposition describes the natural decomposition of the brace operations under the BV operator in the $\textit{associative}$ case.

\begin{proposition}\label{assocdecomp}  Let $\beta_n \in \op{TS}(n+1)$ and $\delta \in \op{TS}(1)$ be as above.  For a tree or collection of trees $\tau$ let $\overline{\tau}$ represent the chain complex generated by $\tau$ and the differential $\partial_T$.  Then
\begin{equation*}
\overline{\beta_n\circ_1\delta}\cong CC_\ast(\Delta_n\times S^1)
\end{equation*}
\end{proposition}
\begin{proof}  As above there is a homeomorphism $\Delta_n\times S^1\cong |\overline{\beta_n}|\times|\overline{\delta}|$.  The composition $\beta_n \circ_1 \delta$ is a sum of trees that can be formed by grafting the branches of $\beta_n$ on to $\delta$.  Let $v$ be the vertex of $\beta_n$ labeled by $1$.  Choosing a gluing scheme in a particular case is equivalent to choosing where to put the spine in relation to the branches.  As such the trees found in the composition can be given by rotating the spine (which starts at the base point) around $v$.  In other words
\begin{equation}\label{compeq}
\beta_n\circ_1\delta = \ds\sum_{l=0}^n(-1)^{l}\beta_{n,l}
\end{equation}
There are $n+1$ possible choices for where to put the spine, which correspond to the open cells $D_i(int(\Delta_{n+1}))$.  There are then $n$ codimension $1$ cells which correspond to the spine being placed on an edge on the one side and to $D_i(\Delta_{n+1})\cap D_{i+1}(\Delta_{n+1})$ (taken mod $n+1$) on the other.  See Figures $\ref{fig:brace}$ and $\ref{fig:triangle}$.
\end{proof}
The important feature to notice is that while the BV operator is induced by $S^1$ on the topological level and the brace operations are induced by simplicies $\Delta_n$, the cellular structure of their composition is not the product structure.
\begin{figure}[htbp]
	\centering
		\includegraphics[scale=.4]{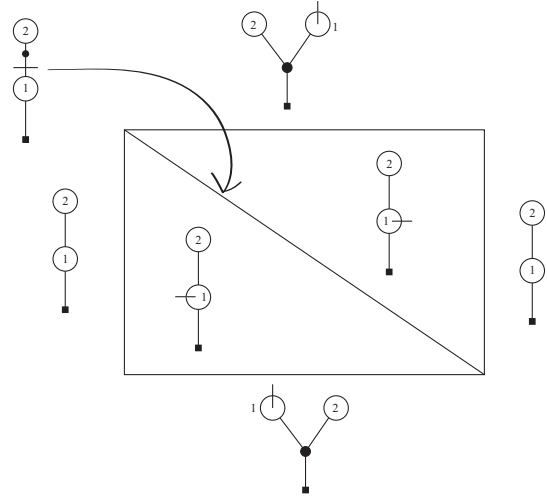}
			\caption{ $\beta_1\circ_1\delta$ gives a cellular decomposition of the cylinder $\Delta_1\times S^1$.}
	\label{fig:brace}
\end{figure}

\begin{figure}[htbp]
	\centering
		\includegraphics[scale=.3]{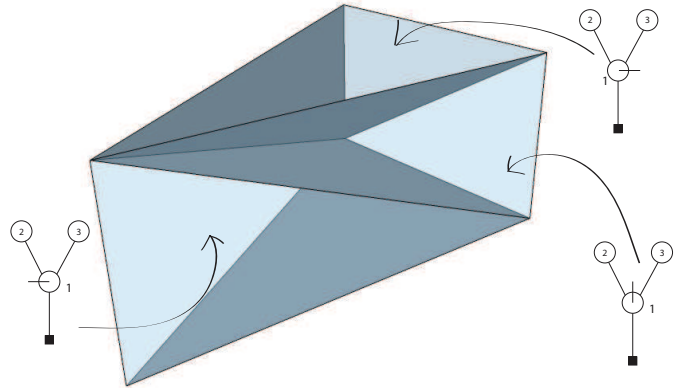}
	\caption{$\beta_2\circ_1\delta$ gives a cellular decomposition of $\Delta_2\times S^1$.  The three trees pictured correspond to the three $3$-cells of the decomposition.}
	\label{fig:triangle}
\end{figure}

\subsection{The $A_\infty$ case}

A fundamental result of \cite{KSch} is that the minimal operad of \cite{KS} is topological.  More precisely there is a collection of CW complexes whose cellular chains are $\op{T}_\infty$.  I will call these spaces $|\op{T}_\infty|$.  Define the space $|\op{TS}_\infty(n)|:=|\op{T}_\infty(n)| \times (S^1)^n$.  Once again the dg-operad $\op{TS}_\infty$ is not given by the product of these two cell structures, but there is a cell structure on $|\op{TS}_\infty(n)|$ which does give $\op{TS}_\infty(n)$ as we now describe.

\subsubsection{Compatibility of the BV Operator and the Brace Operations} 
In the non-associative case the brace operations are governed by cyclohedra.  Recall that the cyclohedron $W_n$ is an abstract polytope whose vertices correspond to full cyclic bracketings of $n$ letters and whose codimension $m$ faces are those partially bracketed expressions with $m$ pairs of brackets.  The dimension of $W_n$ is $n-1$ and its top dimensional cell is indexed by the empty bracketing.  Consider $W_n$ as a CW complex with the canonical CW structure.  Then we have

\begin{proposition} \cite{KSch}  Let $\overline{\beta}$ be the chain complex generated by $\beta$ and $\partial_H$.  As chain complexes $\overline{\beta_n}\cong CC_\ast(W_{n+1})$.
\end{proposition}
In particular if we consider $\beta_2$, the space $|\overline{\beta_2}|$ is a $2$-simplex in the associative case, which in the nonassociative case is blown up at each vertex to form a hexagon, whose 6 sides correspond to the 6 cyclic bracketing of a string a $3$ letters.  See figure 5 of \cite{KSch}.  Considering the brace operation under the BV operator, $\beta_2\circ_1\delta$ is again a blow up of the associative case (pictured in Figure $\ref{fig:triangle}$) which decomposes $W_3\times S^1$ in to $3$ pieces.  See Figure $\ref{fig:blowup}$.  Thus we see:
\begin{lemma} As topological spaces $|\overline{\beta_n\circ_1\delta}|\cong W_{n+1}\times S^1$.
\end{lemma}
\begin{figure}[htbp]
	\centering
		\includegraphics[trim = 5mm 0mm 0mm 0mm,scale=.45]{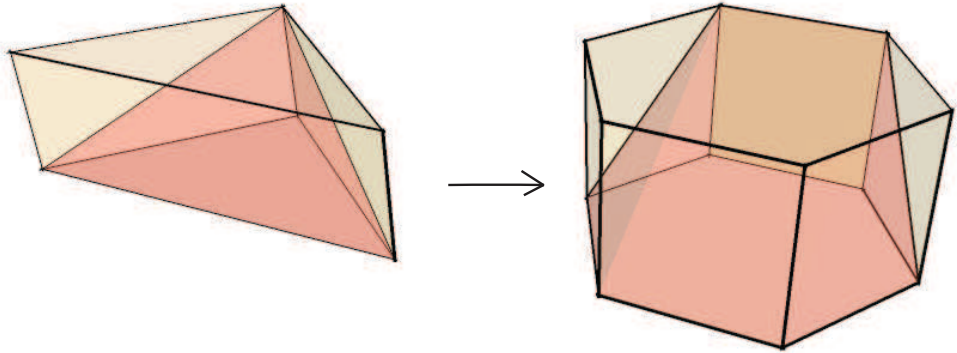}
	\caption{On the left is $\beta_2\circ_1\delta$ in the associative case (with one layer not drawn).  Blowing up gives a cellular decomposition of $W_3\times S^1$.}
	\label{fig:blowup}
\end{figure}

In \cite{KSch} the authors define a sequential blow-up of the standard $n$-simplex $\Delta_n$ to the cyclohedron $W_{n+1}$ achieved by a combinatorially described replacement of certain vertices and faces of $\Delta_n$ by products of cyclohedra and associahedra.  We can give yet another description of this blowup using our tree language as follows.  Label the cells of the $n$ simplex by spineless trees as suggested by the brace operations.  In particular $\beta_n$ labels the top dimensional cell and the vertices are labeled with corollas labeled left to right by $2\cdc i,1,i+1\cdc n+1$.  Then the blow up is achieved by labeling each nontrivial black vertex of arity $n$ by $\mu_n$.  The vertices are blown up to associahedra $K_{n+1}$ and other faces are blown up to compositions (products) of cyclohedra and associahedra which are determined by the combinatorics of the tree labeling said face.  In particular black vertices appearing in a tree correspond to associahedra and internal white vertices correspond to cyclohedra.  See Figures $\ref{fig:cyclo2}$ and $\ref{fig:cyclo1}$.  Recall the cellular structure of $\Delta_n\times S^1$ given in Proposition $\ref{assocdecomp}$.  This blow-up induces a cellular structure on $W_{n+1}\times S^1$, which we take as the cellular chains of $W_{n+1}\times S^1$.

\begin{figure}
	\centering
		\includegraphics[scale=.7]{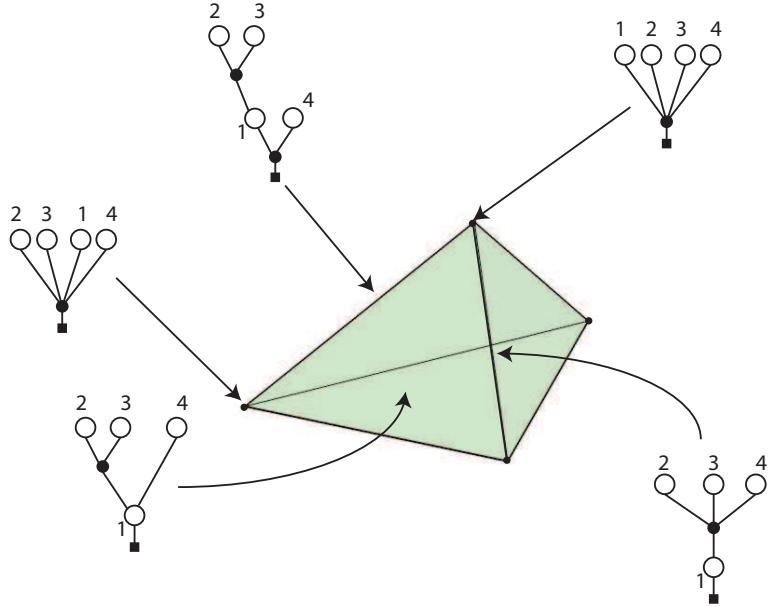}
	\caption{In the associative case the tree $\beta_3$ generates a three simplex whose boundary cells are labeled by trees without spines.  The blow up to the non-associative case is achieved by adding labels to the black vertices, see Figure $\ref{fig:cyclo1}$.}
	\label{fig:cyclo2}
\end{figure}

\begin{figure}
	\centering
		\includegraphics[scale=.7]{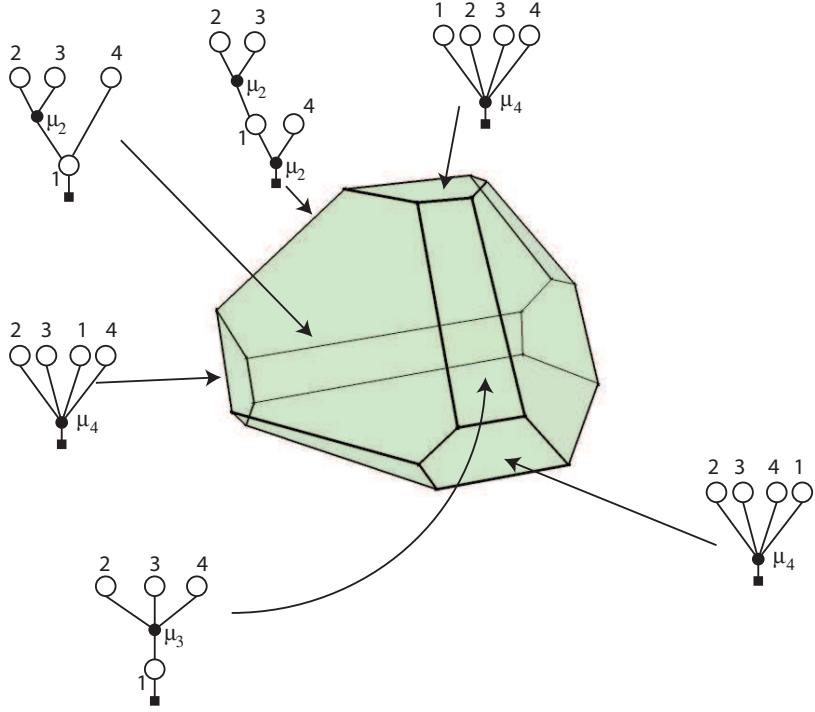}
	\caption{Blowing up $\Delta_3$ in Figure $\ref{fig:cyclo2}$ gives the cyclohedron $W_4$ labeled by trees in $\op{T}_\infty$.  Notice that the boundary cells are products of associahedra and cyclohedra as dictated by the labeling trees.} 
	\label{fig:cyclo1}
\end{figure}

\begin{proposition} As chain complexes $(\overline{\beta_n\circ_1\delta},\partial_H)\cong CC_\ast(W_{n+1}\times S^1)$.
\end{proposition}
\begin{proof}  It is enough to show that the two differentials coincide.  We know that the result holds before blow-up of $\Delta_n\times S^1\to W_{n+1}\times S^1$ by Proposition $\ref{assocdecomp}$.  On the one hand, this blow-up is achieved by labeling each black vertex of a tree labeling $\Delta_n\times S^1$ by some $\mu_r$, with $r\geq 2$.  On the other hand, the terms of the differential $\partial_T$ correspond to those trees labeled with $\mu_2$ and the additional terms in the differential $\partial_H$ (not appearing in $\partial_T$) correspond exactly to those trees which can be formed with higher associahedra $(r \geq 3)$.
\end{proof}

\subsection{Definition of $\op{X}$}  We can now define the collection of CW complexes $\op{X}=\{X_n\}$.  As spaces $X_n = |\op{T}_\infty(n)|\times (S^1)^n$.  Notice that the cellular decomposition of $W_{n+1}\times S^1$ described above has $n+1$ top dimensional cells corresponding to the trees $\beta_{n,l}$ and two adjacent top dimensional cells are glued along a codimension one face corresponding to the trees $\beta_n\wedge_l\mu_2$.  Since the brace operations correspond to top dimensional cells of cyclohedra and the corollas correspond to cells of associahedra, we have explicit cellular descriptions of each of the generators of $\op{TS}_\infty$.  For a tree $\tau$ which is a composite of generators, decompose $\tau$ into a product of generators $\tau_i$ such that each operadic composition is simple (see Lemma $\ref{generators}$).  Then if $c(\tau)$ represents the cell associated to $\tau$ we define
\begin{equation*}
c(\tau)=\times_i c(\tau_i)
\end{equation*}
Then define $X_n$ to be the CW complex whose cells correspond to marked trees with $n$ white vertices and with the cellular differential coming from $\partial_H$, i.e.
\begin{equation*}
\partial(c(\tau)):=c(\partial_H(\tau))
\end{equation*}
In the spirit of \cite{KSch} and \cite{KCacti} a point in $X_n$ is given by a marked tree in $\op{TS}_\infty(n)$ along with a weight on each arc of a white vertex such that the weights at each white vertex sum to $1$ and along with a point in each cell of associahedra which labels a black vertex.  Such pictures can be drawn either as trees or as parameterized circles whose intersection points are labeled with cells of associahedra and so the spaces $\op{X}$ can be thought of as a homotopy associative version of $\op{C}acti$.  More precisely:

\begin{theorem}\label{mainthm2}  For each $n\geq 1$ there is an isomorphism of dg vector spaces
\begin{equation*}
CC_\ast(X_n;k)\cong (\op{TS}_\infty(n), \partial_H)
\end{equation*}
inducing a bijection between the set of cells and the set of trees in $\op{TS}_\infty(n)$.  Additionally for each $n$ there is a surjective homotopy equivalence
\begin{equation*}
X_n\to Cacti^1(n)
\end{equation*}
and this collection of maps induces a morphism of dg operads on the cellular chains
\begin{equation*}
\op{TS}_\infty\stackrel{\nu}\longrightarrow CC_\ast(Cacti^1)\cong\op{TS}
\end{equation*}
\end{theorem}
\begin{proof}  Define a map $X_n\to Cacti^1(n)$ as follows.  A point in $X_n$ corresponds to a weighted marked tree.  Removing the labels of the black vertices we get a weighted tree with spines, which corresponds to a point in $|\op{TS}(n)|\cong Cacti^1(n)$.  This map has the effect of contracting all associahedra to a point, and by extension contracting cyclohedra to associahedra.  Notice that this map is cellular since it decreases the degree when contracting associahedra.  As such this map induces a map on the cellular chains for each $n$ which takes a marked tree $\tau$ to the associated tree with spines if $\tau$ has only black vertices of degree zero, and takes $\tau$ to $0$ otherwise.  Hence this collection of maps induces a morphism of operads, since the operad structure in $\op{TS}_\infty$ has strictly more terms that the operad structure in $\op{TS}$ which correspond to grafting branches to black vertices, and these additional terms are are mapped to zero by this collection of maps.  Moreover this is a morphism of dg operads since the additional terms in the differential of $\op{TS}_\infty$ all arise by labeling black vertices with labels of nonzero degree.
\end{proof}
\begin{corollary}\label{cmcor}  The map $\op{TS}_\infty\stackrel{\nu}\longrightarrow\op{TS}$ is a quasi-isomorphism.  In particular $\op{TS}_\infty$ is a chain model for $fD_2$.
\end{corollary}
\begin{proof}  Since $\nu$ is induced levelwise by homotopy equivalences it induces a levelwise isomorphism in homology.  This levelwise isomorphism is also a morphism of operads, since $\nu$ is, and hence is an isomorphism of operads.  Then since $\op{TS}$ is a chain model for $fD_2$, so is $\op{TS}_\infty$.
\end{proof}
\begin{remark}  We have not given $\op{X}$ a (quasi)-operad structure on the topological level, since it is not necessary for our purposes.  However it should be possible to do this using trees with height in the spirit of \cite{KSch}.  The result would be be a topological `operad' where the operadic composition is only associative up to rescaling.  A possible future direction would be to investigate coherence laws for such quasi-operads and their higher operadic counterparts \cite{bat} in the context of higher versions of Deligne's conjecture (see e.g. \cite{DTT}).
\end{remark}

\subsubsection{Orientation and Signs}\label{diffsigns}  The signs in the operadic composition and the differential $\partial_H$ come most naturally from an orientation of the CW complexes $X_n$ as we now explain.  We first fix an orientation for each cell in $X_n$ which we call the standard orientation.

First for corollas we take the natural orientation of associahedra indicated by the differential.  In other words there is a unique orientation of each associahedron such that the cellular differential coincides with the the differential relations for $A_\infty$ algebras (see e.g. \cite{MSS} p.195).  For a tree $\beta_n$ encoding a brace operation we recall that in the associative case $|\overline{\beta_n}|$ is an $n$ simplex whose vertices correspond to labelings of corollas by $(2,3\cdc i-1,1,i\cdc n+1)$.  We call the vertex with this label $v_i$ and take as the standard orientation of $\beta_n$ that induced by the ordering of vertices $(v_1\cdc v_{n+1})$.  We then define the standard orientation of cyclohedra inductively as the orientation induced by the blowup.  In particular any cell appearing in the boundary of a cyclohedron is an ordered product of associahedra and lower cyclohedra each of which is oriented, so the product is oriented.  Finally for any spines we take the standard orientation as clockwise on each white vertex.  In particular the orientation of $\beta_m\wedge_l \mu_n$ is induced by the projection down to $\beta_m$ and the orientation of $\beta_{m,l}$ agrees with (resp. disagrees with) the orientation induced by $\beta_{m}\times S^1$ if $l$ is even (resp. odd).  Cells which are operadic compositions of generators are oriented as the ordered product of oriented cells given by taking the decomposition in to generators in Lemma $\ref{generators}$ in the order induced by the total order on flags (i.e. traversing the tree clockwise starting from the root).  Finally, we stipulate that the $S_n$ action does not change the orientation of a cell.

The orientation of $X_n$ given above fixes the signs in the differential $\partial_H$.  In particular given a tree $\tau$, the differential $\partial_H(\tau)$ is a signed sum of trees which come with a standard orientation.  These cells come also with an orientation induced from the standard orientation of $\tau$.  If these orientations agree the sign in the sum is $+$ and if they disagree it is $-$.  Notice also that our description of $\partial_H$ as the cellular differential of a CW complex assures us that $\partial_H^2=0$.

For the operadic composition the signs can be determined in a similar manner.  In particular $|\overline{\tau_1\circ_i\tau_2}|$ is a subcomplex of some $X_n$ and a choice of orientation of the entire subcomplex determines the signs of the trees appearing in $\tau_1\circ_i\tau_2$ again by comparison with the standard orientation of each top dimensional cell.  We take as a convention the orientation of the subcomplex $|\overline{\tau_1\circ_i\tau_2}|$ induced by the top dimensional cell (tree) formed by gluing each branch at the largest possible angle (starting from the root and going clockwise).  See Figure $\ref{fig:sings}$.

\begin{figure}[htbp]
	\centering
		\includegraphics[scale=.95]{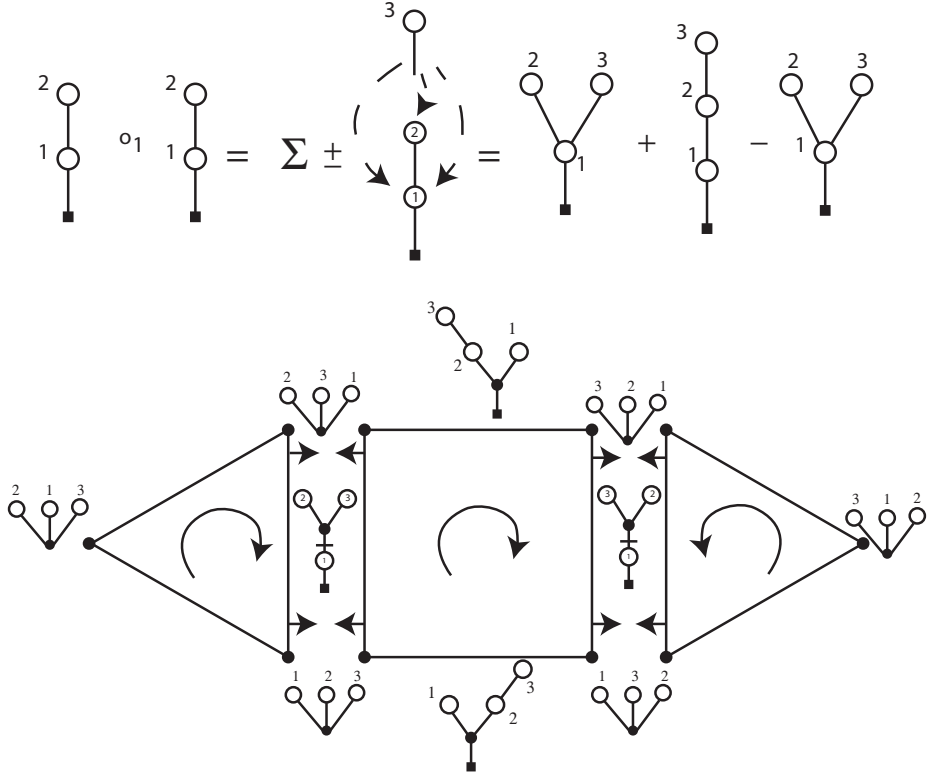}
		\caption{A basic example of the operadic composition with signs.  The first tree on the right hand side $\beta_2$ determines the orientation of the entire subcomplex $|\overline{\beta_1\circ_1\beta_1}|$ since this tree takes each (one) branch glued at the largest possible angle.  Then the orientations agree for the first two trees and disagree for the third, hence the signs.  The orientation for the second tree is given by decomposing it as $\beta_1\circ_2\beta_1$ and taking the orientation induced by this ordered product which is up then right.}
	\label{fig:sings}
\end{figure}

\section{The Main Theorem}\label{sec:mainthm}  Piecing together the above work we can now prove the main theorem.  Let $D_2$ be the little disks operad and $fD_2$ be the framed little disks operad.

\begin{theorem} \label{mainthm} (Cyclic $A_\infty$ Deligne Conjecture)  Let $A$ be a cyclic $A_\infty$ algebra.  There is a chain model for the operad $fD_2$ that acts on the Hochschild cochains of $A$ in a manner compatible with the standard actions on homology/cohomology.
\end{theorem}

\begin{proof}  The chain model is the dg operad $(\op{TS}_\infty, \partial_H)$.  The action is defined in subsection $\ref{actionsection}$. That the action is dg is a consequence of Theorem $\ref{diffthm}$.  In addition Corollary $\ref{cmcor}$ tells us that $\op{TS_\infty}$ is a chain model for $fD_2$, and that the action preserves the standard operations, in particular the multiplication, the Gerstenhaber bracket, and the BV operator.  
\end{proof}

\begin{remark}
In \cite{KSch} Kaufmann and Schwell construct a dg operad which is a chain model for the little disks and which is isomorphic to the minimal operad of Kontsevich and Soibelman given in \cite{KS}.  Since the minimal operad acts on the Hochschild cochains of an $A_\infty$ algebra, this gives a proof of a (non-cyclic) $A_\infty$ version of Deligne's conjecture.  By Lemma $\ref{minlemma}$ we can recover their result by restricing Theorem $\ref{mainthm}$ to the suboperad of spineless trees $\op{T}_\infty\subset \op{TS}_\infty$.
\end{remark}

\subsection{An alternate chain model: $\op{BV}_\infty$}  In \cite{GCTV} the authors construct a dg operad $\op{BV}_\infty$ in characteristic $0$ which is an explicit cofibrant replacement of the operad $H_\ast(fD_2)\cong \op{BV}$ in the model category of dg operads \cite{MC1},\cite{MC2}.  Since we have shown that there is a zig-zag of quasi-isomophisms from $\op{TS}_\infty$ to $C_\ast(fD_2)$, the singular chains, we immediately have the following result.

\begin{theorem}\label{tvthm}  Let $k$ be of characteristic 0.  There is a quasi-isomorphism of dg operads
\begin{equation*}
\op{BV}_\infty\stackrel{\sim}\to\op{TS}_\infty
\end{equation*}
and hence $\op{BV}_\infty$ satisfies the cyclic $A_\infty$ Deligne conjecture in characteristic 0.
\end{theorem}
\begin{proof}  The proof of this fact follows similarly to the proof in the associative case given in \cite{GCTV}.  Since $k$ is of characteristic zero, $C_\ast(fD_2)$ is formal \cite{GS},\cite{Sev}, i.e. there is a zig-zag of quasi-isomorphisms 
\begin{equation*}
C_\ast(fD_2)\stackrel{\sim}\longleftarrow\dots \stackrel{\sim}\longrightarrow \op{BV}
\end{equation*}
and hence there is a zig-zag of quasi-isomorphisms
\begin{equation*}
\op{TS}_\infty\stackrel{\sim}\longleftarrow\dots \stackrel{\sim}\longrightarrow \op{BV}
\end{equation*}
Let $X$ be a cofibrant replacement of $\op{TS}_\infty$.  Then since $X$ is cofibrant and $\op{BV}$ is fibrant (every operad is fibrant in this model category), there is a quasi-isomorphism $X\stackrel{\sim}\longrightarrow \op{BV}$.  Since $\op{BV}_\infty$ is cofibrant and $X$ is fibrant, the diagram
\begin{equation*}
\op{BV}_\infty \stackrel{\sim}\longrightarrow \op{BV} \stackrel{\sim}\longleftarrow X 
\end{equation*}   
induces a morphism $\op{BV}_\infty\stackrel{\sim}\longrightarrow X$ which when composing gives a morphism
\begin{equation}\label{bvinftymor}
\op{BV}_\infty \stackrel{\sim}\longrightarrow \op{TS}_\infty
\end{equation}
For a cyclic $A_\infty$ algebra $A$ we can compose this morphism with the action given above to get a morphism,
\begin{equation}\label{bvinftyaction}
\op{BV}_\infty \to End_{CH^\ast(A,A)}
\end{equation}
which shows that $\op{BV}_\infty$ satisfies the cyclic $A_\infty$ Deligne conjecture.
\end{proof}

\begin{remark}  As discussed in the introduction, the morphism in Theorem $\ref{tvthm}$ is not constructed explicitly but is the result of an abstract model category argument.  An explicit construction of this morphism seems difficult, however I will outline a process by which this morphism could be constructed in a manner which is somewhat explicit.  The operad $\op{BV}_\infty$ can be described it terms of a set of generators $\textbf{M}:=\{m_{p_1\cdc p_t}^d\}$ subject to a family of relations $\textbf{R}:=\{R_{p_1\cdc p_t}^d\}$.  See Theorem 20 of \cite{GCTV}.  Give these sets a total order lexicographically, starting with $d$, then $t$, then $p_t$, then $p_{t-1}$, etc.  Then a given relation $R_{p_1\cdc p_t}^d$ is defined only in terms of the elements of $\textbf{M}$ which are less than or equal to $m_{p_1\cdc p_t}^d$.  Now to construct the morphism $\op{BV}_\infty\stackrel{\phi}\to\op{TS}_\infty$, one must define $\phi$ on generators such that the relations are sent to zero.  Start by defining $\phi(m^0_n)=\mu_n$, $\phi(m^0_{1,1})$ is the Gerstenhaber bracket, $\phi(m_{p_1\cdc p_t}^0)$ is the $\op{G}_\infty$ structure as in \cite{VorG},\cite{TT} and $\phi(m_1^1)$ is the BV operator.  Moreover the topological framework given above can be exploited to determine explicit homotopies which induce the BV structure on homology.  For example one can determine $\phi(m^1_2)$ from figure 12 of \cite{KLP}.  After defining $\phi$ for the first `several' generators such that the corresponding relations are sent to zero, it should be possible to devise an arguement by (trans-finite) induction that $\phi$ extends to all of $\op{BV}_\infty$.  Then since the induced map on homology takes generators to generators it will be a quasi-isomorphism.  
\end{remark}

\section{Extension to Cyclic $A_\infty$ Categories.}\label{section:cat}  The purpose of this section is to extend Deligne's conjecture (Theorem $\ref{mainthm}$) to (unital) cyclic $A_\infty$ categories (also known as Calabi-Yau $A_\infty$ categories in \cite{Co}).  In particular we will prove the following theorem:

\begin{theorem}\label{catthm}  Let $\op{A}$ be a cyclic $A_\infty$ category.  Then $CH^\ast(\op{A},\op{A})$ is an algebra over the operad $\op{TS}_\infty$.
\end{theorem}

The remainder of this section will be devoted to first recalling relevant particulars on $A_\infty$ categories and cyclic $A_\infty$ categories, and to the proof of Theorem $\ref{catthm}$.

\subsection{Background} The purpose of this section is to give background on $A_\infty$ categories and to illustrate how, under certain mild assumptions, the study of $A_\infty$ categories and their modules, as well as $A_\infty$ functors can be reduced to the study of $A_\infty$ algebras and their modules and morphisms.  

\subsubsection{$A_\infty$ Categories}
\begin{definition}  A dg pre-category $\op{C}$ is a collection of objects $Ob(\op{C})$ and a collection of dg vector spaces indexed over all ordered pairs in $Ob(\op{C})\times Ob(\op{C})$.  The dg vector space associated to a pair of objects $(X,Y)$ is denoted $Hom_\op{C}(X,Y)$.  A dg pre-category will be called small if its collection of objects is a set.
\end{definition}

In order to simplify the presentation we use the following notation.  If $\op{A}$ is an dg pre-category and $X_1\cdc X_{n+1}$ is a list of objects, we write
\begin{equation*}
\op{A}(X_1\cdc X_{n+1}):= Hom_\op{A}(X_1,X_2)\tdt Hom_\op{A}(X_n,X_{n+1})
\end{equation*}

\begin{definition}  An $A_\infty$ category $\op{A}$ is a dg pre-category along with maps $\mu_n$
\begin{equation*}
\op{A}(X_1\cdc X_{n+1})\stackrel{\mu_n}\longrightarrow\op{A}(X_1,X_{n+1})
\end{equation*}
for every $n\in\N$ and every ordered collection of $n+1$ objects $(X_1\cdc X_{n+1})$, satisfying the equation
\begin{equation*}
\mu\circ\mu=0
\end{equation*}
where $\mu$ is as in equation $\ref{mueq}$.  In particular $\mu_1$ is the differential.  Note here that the signs are built into the $\circ$ operation, as in equation $\ref{preliesigns}$.  From now on we will consider only `unital' $A_\infty$ categories, in the sense of Remark $\ref{unitremark}$.  This means that we have unit morphisms in each $Hom(X,X)$, and that the higher $\mu_n$ vanish when evaluated at a product with an identity factor.
\end{definition}

\begin{definition}  An $A_\infty$ functor between two $A_\infty$ categories $\op{A}$ and $\op{B}$ is a map $Ob(\op{A})\stackrel{\phi}\longrightarrow Ob(\op{B})$ along with a collection of linear maps $\phi_n$,
\begin{equation*}
\op{A}(X_0\cdc X_n)\stackrel{\phi_n}\longrightarrow\op{B}(\phi(X_0),\phi(X_n))
\end{equation*}
satisfying
\begin{equation*}
\phi_\ast\circ\mu = \sum_{n}\sum_{n=i_1+\dots+i_l}\pm\mu_l(\phi_{i_1}\tdt \phi_{i_l})
\end{equation*} where $\phi_\ast$ is the direct sum.
Where the we take the sign conventions as in \cite{LH}.  An $A_\infty$ functor will be called strict if $\phi_n=0$ for $n\geq 2$.  Notice that a strict $A_\infty$ functor resembles a (proper) functor in that the data is a correspondence of objects and morphisms such that composition is preserved (even though composition need not be associative).
\end{definition}
Given two $A_\infty$ functors $\phi\colon\op{A}\to\op{B}$ and $\psi\colon\op{B}\to\op{C}$ we can compose them by composing on objects and by taking
\begin{equation*} (\psi\circ\phi)_n:=\sum_{i_1+\dots+i_l=n}\pm\psi_l(\phi_{i_1}\tdt\phi_{i_l})
\end{equation*}
This composition is strictly associative \cite{Seidel}.  Note also that there an obvious identity $A_\infty$ functor from an $A_\infty$ category to itself.  Thus the collection of $A_\infty$ categories can be made into a (proper) category.  

\begin{definition}  Let $\textbf{Cat}_\infty$ be the category of small $A_\infty$ categories of finite type.  
\end{definition}

Here finite type means that $Hom(X,Y)$ is finite dimensional for each pair $(X,Y)$.  The assumptions on our $A_\infty$ categories given in this definition are not yet necessary, but will be used below.  The smallness assumption will be necessary for the adjunction we consider.  The finite type assumption will be necessary in the cyclic case when we consider Hochschild cohomology.

The above framework can be seen as a generalization of the theory of $A_\infty$ algebras.  In particular we can consider an $A_\infty$ algebra $A$ as an $A_\infty$ category with one object $X$ by setting $Hom(X,X):=A$.  Then the standard notion of $A_\infty$ algebra morphism and strict morphism coincide with the above definitions.  We write $\textbf{Alg}_\infty$ for the category of unital $A_\infty$ algebras and morphisms, and write $\fr{i}$ for the inclusion functor:
\begin{equation*}
\textbf{Alg}_\infty\stackrel{\fr{i}}\hookrightarrow\textbf{Cat}_\infty
\end{equation*}
Note that the unital assumption is necessary for the associated $A_\infty$ category to have a identity morphism.

\subsubsection{Adjointness} 

There is an obvious candidate for a left adjoint to $\fr{i}$ defined as follows.  Let $\op{A}$ be a small $A_\infty$ category, and define a (possibly infinite dimensional) dg vector space $\op{A}_\oplus$
\begin{equation*}
\op{A}_\oplus:=\bigoplus_{(X,Y) \in Ob(\op{A})^{\times 2}}Hom_\op{A}(X,Y)
\end{equation*}
\begin{lemma} $\op{A}_\oplus$ is naturally a unital $A_\infty$ algebra.
\end{lemma}
\begin{proof}  Write $\mu_n$ for the $A_\infty$ category structure maps.  To give an $A_\infty$ algebra structure to $\op{A}_\oplus$ we want to define maps $m_n\colon [\op{A}_\oplus]^{\tensor n}\to \op{A}_\oplus$, and by linearity it is enough to define them on
\begin{equation*}
Hom(X_1,X_2)\tensor Hom(X_3,X_4) \tdt Hom(X_{2n-1},X_{2n})\stackrel{m_n}\longrightarrow Hom(X_1,X_{2n})
\end{equation*}
which is done as follows.  If $X_{2i}=X_{2i+1}$ for all $i=1\cdc n$, define $m_n:=\mu_n$.  Otherwise, define $m_n=0$.  Then the $A_\infty$ relation $\mu\circ\mu=0$ implies that $m\circ m=0$.
\end{proof}

\begin{lemma} The assignment $\textbf{Cat}_\infty\to\textbf{Alg}_\infty$ given by $\op{A}\mapsto \op{A}_\oplus$ is functorial.
\end{lemma}
By abuse of notation we will call this functor $\oplus$ and write $\oplus(\op{A})=\op{A}_\oplus$:
\begin{equation*}
\xymatrix{\textbf{Cat}_\infty   \ar@/^/[rr]^{\oplus}&& \ar@/^/[ll]^{\fr{i}} \textbf{Alg}_\infty }
\end{equation*}
Notice that $\oplus\circ\fr{i}= id_{\textbf{Alg}_\infty}$ and that there is an  natural transformation via inclusion $id_{\textbf{Cat}_\infty}\to\fr{i}\circ\oplus$ and so these functors form an adjoint pair.
\begin{remark}  The adjunction given here can also be used to study $A_\infty$ category modules in terms of an associated $A_\infty$ algebra module.  Roughly, an $\op{A}$-module is an $A_\infty$ functor from $\op{A}$ to chain complexes, and given such a functor $\Phi$, we can define
\begin{equation*}  M=\bigoplus_{X\in ob(\op{A})} \Phi(X)
\end{equation*}
Then $M$ is naturally a module over the $A_\infty$ algebra $\op{A}_\oplus$.  Below we will consider the Hochschild cohomology of $\op{A}$ with values in $\op{A}$, but we could also use this approach to consider the Hochschild cohomology with values in a module.
\end{remark}

\subsubsection{The cyclic case.}  

\begin{definition}\label{cyccatdef}  A (unital) cyclic $A_\infty$ category is a small $A_\infty$ category of finite type along with a symmetric nondegenerate pairing
\begin{equation*}
Hom(X,Y)\tensor Hom(Y,X)\stackrel{\langle-,-\rangle}\longrightarrow k
\end{equation*}
for each pair of objects $X,Y \in ob{A}$ which satisfies the following property:  if we extend $\langle-,-\rangle$ to all of $[\op{A}_\oplus]^{\tensor 2}$ by zero, then $(\op{A}_\oplus, \langle-,-\rangle)$ is a (unital) cyclic $A_\infty$ algebra.  This property simply means that equation $\ref{eqcyc}$ holds in the appropriate categorical sense.
\end{definition}

Here is where we use our finite type assumption: for a cyclic $A_\infty$ category $\op{A}$, there is a canonical map $Hom_\op{A}(X,Y)\to Hom_\op{A}(Y,X)^\ast$ given by $f \mapsto \langle f,- \rangle$.  Nondegeneracy of the pairing implies that this map is an injection.  Since the same procedure produces an injection if we switch $X$ and $Y$, we see that $dim(Hom_\op{A}(X,Y))=dim(Hom_\op{A}(Y,X))$, and thus the canonical injection is an isomorphism.  Extending linearly we have the following:
\begin{lemma}  Let $\op{A}$ be a cyclic $A_\infty$ category.  Then $\op{A}_\oplus$ is of finite type and there is a canonical isomorphism $\op{A}_\oplus\to [\op{A}_\oplus]^\ast$.
\end{lemma}

\subsubsection{Hochschild cohomology}  Let $\op{A}$ be an $A_\infty$ category.  We will recall the definition of $CH^\ast(\op{A},\op{A})$.  Define the graded vector space $CH^\ast(\op{A},\op{A})$ by
\begin{equation*}
CH^\ast(\op{A},\op{A})=\bigoplus_{n\geq1}\left[\ds\coprod_{(X_1\cdc X_{n+1})}Hom(\op{A}(X_1\cdc X_{n+1}),\op{A}(X_1,X_{n+1}))\right]
\end{equation*}
with the total grading, i.e. if 
\begin{equation*}
f\in Hom(\op{A}(X_1\cdc X_{n+1}),\op{A}(X_1,X_{n+1}))
\end{equation*}
has degree $deg(f)$, then, if $|f|$ denotes the grading of $f\in CH^\ast(\op{A},\op{A})$, we define $|f|=deg(f)+n$.  Notice that there is a canonical isomorphism
\begin{equation*}
CH^\ast(\op{A}_\oplus,\op{A}_\oplus)\cong\bigoplus_{\substack{n\geq 1 \\ (X_1,\cdc X_{2n}) \\ (Y,Z)}} Hom\left[ \op{A}(X_1,X_2)\tensor \op{A}(X_3,X_4)\tdt \op{A}(X_{2n-1},X_{2n}),\op{A}(Y,Z)\right]
\end{equation*}
and as such we can consider 
\begin{equation*}
CH^\ast(\op{A},\op{A})\subset CH^\ast(\op{A}_\oplus,\op{A}_\oplus)
\end{equation*}
where $CH^\ast(\op{A},\op{A})$ consists of those summands such that $X_{2i-1}=X_{2i}$ for $1 \leq i \leq n$ and $X_1=Y$ and $X_{2n}=Z$.  Note that this subspace is closed under the $A_\infty$ algebra differential $d$ defined above. As such, we can define a differential on $CH^\ast(\op{A},\op{A})$ induced by inclusion.  This cochain complex will be called the Hochschild cochains of $\op{A}$ and its cohomology is the Hochschild cohomology of $\op{A}$, denoted $HH^\ast(\op{A},\op{A})$.

\begin{remark}  Defining the Hochschild cohomlogy of $\op{A}$ in terms of that of $\op{A}_\oplus$ is non-standard, but it agrees with and simplifies the standard presentation.  Although we are concerned here with Hochschild cohomology of $A_\infty$ categories, this approach could also be used to consider their Hochschild homology.
\end{remark}
  
\subsection{Proof of theorem $\ref{catthm}$}  Roughly speaking the operadic action which establishes this theorem is given by applying the action to the relevant direct summand of the associated cyclic $A_\infty$ algebra.

As above, we consider
\begin{equation*}
CH^\ast(\op{A},\op{A})\subset CH^\ast(\op{A}_\oplus,\op{A}_\oplus)
\end{equation*}
which gives us maps
\begin{equation*}
CH^\ast(\op{A},\op{A})^{\tensor n}\tensor \op{TS}_\infty(n)\to CH^\ast(\op{A}_\oplus,\op{A}_\oplus)
\end{equation*}
and it is enough to show that the image actually lands in $CH^\ast(\op{A},\op{A})$.  To see this it is enough to check the action of the generators of $\op{TS}_\infty$ on homogeneous cochains.

\begin{lemma}  $\rho(CH^\ast(\op{A},\op{A})^{\tensor n}\tensor \mu_n)\subset CH^\ast(\op{A},\op{A})$
\end{lemma}
\begin{proof}  It is enough to check on homogeneous generators of $CH^\ast(\op{A},\op{A})$.  As such, for each $1\leq j \leq n$ pick
\begin{equation*}
f_j\in Hom(\op{A}(X^j_1\cdc X^j_{m_j}),\op{A}(X^j_1, X^j_{m_j}))
\end{equation*}
then $\rho(f_1\tdt f_n \tensor \mu_n)$ is defined via the following composition:
\begin{equation*}
\op{A}(X^1_1\cdc X^1_{m_1})\tdt\op{A}(X^n_1\cdc X^n_{m_n})\stackrel{f_1\tdt f_n}\longrightarrow \op{A}(X^1_1, X^1_{m_1})\tdt\op{A}(X^n_1, X^n_{m_n}) \stackrel{\mu_n}\longrightarrow \op{A}(X_1^1,X^n_{m_n}) 
\end{equation*}
Now by definition this composition is zero unless $X^j_{m_j}=X^{j+1}_1$ for all $1\leq j \leq n-1$, which in either case implies that $\rho(f_1\tdt f_n \tensor \mu_n)\in CH^\ast(\op{A},\op{A})$.
\end{proof}

\begin{lemma}  Let $f,g\in CH^\ast(\op{A},\op{A})$ be homogeneous elements with $r$ and $s$ inputs respectively.  Then $f\circ_i g \in CH^\ast(\op{A},\op{A})$ for $1\leq i \leq r$.
\end{lemma}
\begin{proof}  Let
\begin{equation*}
f\in Hom(\op{A}(X_1\cdc X_{r+1}),\op{A}(X_1, X_{r+1}))
\end{equation*}
and
\begin{equation*}
g\in Hom(\op{A}(Y_1\cdc Y_{s+1}),\op{A}(Y_1, Y_{s+1}))
\end{equation*}
Now, the $\circ_i$ operation is preformed on tensor powers of the $A_\infty$ algebra $[\oplus_{X,Y} Hom_{\op{A}}(X,Y)]$, by inserting into the $i^{th}$ factor.  After factoring out the direct sums, in order to get a nonzero term we need a summand in which both $g$ and $f$ are nonzero.  This happens only if $Y_1=X_i$ and $Y_{s+1}=X_{i+1}$.  As a result $f\circ_i g \in CH^\ast(\op{A},\op{A})$.
\end{proof}

\begin{corollary}  $\rho(CH^\ast(\op{A},\op{A})^{\tensor n}\tensor \beta_n)\subset CH^\ast(\op{A},\op{A})$
\end{corollary}
\begin{proof}  Recall $\beta_n$ corresponds to a brace operation, which is a sum of $\circ_i$ operations, each one of which is in $CH^\ast(\op{A},\op{A})$ by the previous lemma.
\end{proof}

This shows the minimal operad acts on the Hochschild cochains of an $A_\infty$ category.  We will from now assume that $\op{A}$ is a cyclic $\op{A}_\infty$ category.  There is a unit $1\in\op{A}_\oplus=\oplus Hom(X,Y)$ given by taking the identity morphism in the summands with $X=Y$ and zero else.  

Let $f$ be a homogeneous cochain in $Hom(\op{A}(X_1\cdc X_{n+1}),\op{A}(X_1,X_{n+1}))$.  Define a cochain 
\begin{equation*}
\hat{f}\in Hom(\op{A}(X_{n+1},X_1,X_2\cdc X_{n+1}),\op{A}(X_{n+1},X_{n+1}))
\end{equation*}
implicitly by 
\begin{equation*}
\langle \alpha_0, f(\alpha_1\tdt \alpha_n) \rangle = \langle 1, \hat{f}(\alpha_0\tdt \alpha_n) \rangle
\end{equation*}
where $\alpha_i\in \op{A}(X_i,X_{i+1})$, (mod $n+1$).  Notice that $\hat{f}\in CH^\ast(\op{A},\op{A})$.

\begin{lemma}  $CH^\ast(\op{A},\op{A})$ is closed under the action of spined brace operations and the BV operator.
\end{lemma}
\begin{proof}  Recall that the action of the spined brace operations are implicitly defined as the sum of certain $\circ_i$ operations via the procedure given above.  The BV operator is defined as a sum of implicitly defined operations via this procedure.  Since both the implicit defining and the $\circ_i$ operations are closed in $CH^\ast(\op{A},\op{A})$, the lemma follows.
\end{proof}

Thus the $\op{TS}_\infty$-algebra structure on $CH^\ast(\op{A}_\oplus,\op{A}_\oplus)$ restricts to $CH^\ast(\op{A},\op{A})$, from which the theorem follows.

\appendix
\section{Proof of Theorem $\ref{diffthm}$.}\label{appthm2}  In this section we will prove Theorem $\ref{diffthm}$.  As argued above, it is enough to show that the differential is compatible with the action on generators, and this has already been established for spineless generators by Theorem $\ref{spinelessdiff}$.  Thus it is enough to consider the BV operator and the spined brace operations of types 1 and 2.  These calculations can be performed completely symbolically, although signs present a considerable headache.  They can also be performed more intuitively using the action of trees.  Thus as a compromise we will give the symbolic computation with signs in the case of the BV operator, and the more intuitive approach in the case of a spined brace operation.  Throughout this section $A$ is a cyclic $A_\infty$ algebra.

\subsection{The BV operator}
\begin{proposition}\label{bvdiff}  $\rho(\partial_H(\delta))=d_{Int}(\Delta)=0$.
\end{proposition}
\begin{proof}
Since $\partial_H(\delta)=0$ by definition, it is enough to show that $d_{Int}(\Delta)=0$.  Let $f\in Hom(A^{\tensor n},A)$.  From equation $\ref{internaleq}$ we know
\begin{eqnarray*}
d_{Int}(\Delta(f))=\Delta(d(f))-(-1)^{|\Delta|}d(\Delta(f))=\Delta([f,\mu])+[\Delta(f),\mu]\\
=\Delta(f\circ\mu-(-1)^{||f||-1}\mu\circ f)+\Delta(f)\circ\mu-(-1)^{||\Delta(f)||-1}\mu\circ\Delta(f)
\end{eqnarray*} 
hence,
\begin{equation*}
\langle a_0, d_{Int}(\Delta)(f)(\alpha)\rangle =\bigoplus_{r\geq1} \langle a_0, [\Delta(f)\circ\mu_r-(-1)^{||f||}\mu_r\circ\Delta(f)+\Delta(f\circ\mu_r )+(-1)^{||f||}\Delta(\mu_r\circ f)](\alpha_m) \rangle 
\end{equation*}
where $\alpha \in \mathbb{T}A$ and $\alpha_m$ is the summand in degree $m:=n+r-2$.  Here we assume that $|a_i|$ is even for each $i$.  This assumption can be dropped at the expense of adding additional signs via the Koszul sign rule:
\begin{equation*}
[f\tensor g] (a\tensor b) =(-1)^{|a|deg(g)} f(a)\tensor g(b)
\end{equation*}
Now by our normalization assumption, $\langle a_0, \Delta(\mu_r\circ f)(\alpha_m) \rangle=0$ for $r\geq 3$ and a calculation shows it is also zero for $r=2$.  This term is also zero in the case $r=1$ using Remark $\ref{invardiffrmk}$ and the fact that $\mu_1=d_A$.  Next we will calculate for a fixed $r$ with the help of equations $\ref{preliesigns}$ and $\ref{normalbv}$,

\small{\begin{align*}
&(-1)^{(n+1)r}\langle a_0, \Delta(f\circ\mu_r)(\alpha_m)\rangle=\sum_{i=1}^n\langle a_0, (-1)^{(i-1)(r-1)}\Delta(f\circ_i\mu_r)(\alpha_m)\rangle \\
& = \sum_{i=1}^n \langle 1 ,(-1)^{(i-1)(r-1)}(f\circ_i \mu_r)\circ N_{m+1}(a_0\cdc a_m) \rangle \\
&= \sum_{i=1}^n\sum_{j=1}^{m+1} (-1)^{(j-1)m+(i-1)(r-1)}\langle 1 ,(f\circ_i \mu_r)(a_{m-j+2}\cdc a_m,a_0,\cdc a_{m-j+1}) \rangle \\
&= \sum_{i=1}^n\sum_{j=i}^{i+r-1} (-1)^{(j-1)m+(i-1)(r-1)}\langle 1 ,f (a_{m-j+2}\cdc \mu_r(a_{m-j+i}\cdc a_m,a_0\cdc a_{i+r-j-1})\cdc a_{m-j+1}) \rangle \\
&\indent+\sum_{i=1}^n\sum_{j= i+r}^{i-1} (-1)^{(j-1)m+(i-1)(r-1)}\langle 1 ,(f\circ_i \mu_r)(a_{m-j+2}\cdc a_m,a_0,\cdc a_{m-j+1}) \rangle\\
&= \sum_{i=1}^n \sum_{l=1}^{r}(-1)^{(i+l+r+1)m+(i-1)(r-1)}\langle 1 ,f(a_{n+l-i-1}\cdc\mu_r(a_{l+n-1}\cdc a_m,a_0\cdc a_{l-1}),a_l\cdc a_{n+l-i-2}) \rangle \\
&\indent+\sum_{i=2}^n\sum_{j=1}^{i-1} (-1)^{(j-1)m+(i-1)(r-1)}\langle 1 ,f(a_{m-j+2}\cdc a_0\cdc\mu_r(a_{i-j}\cdc a_{i-j+r-1})\cdc a_{m-j+1}) \rangle\\
&\indent+\sum_{i=1}^{n-1}\sum_{j=i+r}^{m+1} (-1)^{(j-1)m+(i-1)(r-1)}\langle 1 ,f(a_{m-j}\cdc \mu_r(a_{m+i-j+1}\cdc a_{m+i-j+r})\cdc a_m,a_0,\cdc a_{m-j-1}) \rangle\\
&= \sum_{l=1}^{r}\sum_{i=1}^n (-1)^{im+lm+rm+m+ir+i+r+1+(n-1)(i-1)}\langle 1 ,f(t_n^{i-1}(\mu_r(a_{l+n-1}\cdc a_m,a_0\cdc a_{l-1}),a_l\cdc a_{l+n-2})) \rangle \\
&\indent+\sum_{i=2}^n\sum_{j=1}^{i-1} (-1)^{(j-1)m+(i-1)(r-1)+(n-1)(j-1)}\langle 1 ,(f\circ t_n^{j-1}(a_0\cdc\mu_r(a_{i-j}\cdc a_{i-j+r-1})\cdc a_m) \rangle\\
&\indent+\sum_{i=1}^{n-1}\sum_{j=i+r}^{m+1} (-1)^{(j-1)m+(i-1)(r-1)+(n-1)(j-r)}\langle 1 ,f\circ t_n^{j-r}(a_0\cdc \mu_r(a_{m+i-j+1}\cdc a_{m+i-j+r})\cdc a_m) \rangle\\
&= \sum_{l=1}^{r}\sum_{i=1}^n (-1)^{im+lm+rm+m+ir+i+r+1+in+i+n+1}\langle 1 ,f(t_n^{i-1}(\mu_r(a_{l+n-1}\cdc a_m,a_0\cdc a_{l-1}),a_l\cdc a_{l+n-2})) \rangle \\
&\indent+\sum_{s=1}^{n-1}\sum_{j=1}^{s} (-1)^{(j-1)m+(s+j-1)(r-1)+(n-1)(j-1)}\langle 1 ,(f\circ t_n^{j-1}(a_0\cdc\mu_r(a_s\cdc a_{s+r-1})\cdc a_m) \rangle\\
&\indent+\sum_{i=1}^{n-1}\sum_{j^\prime=i+1}^{n} (-1)^{(j^\prime-1)m+(i-1)(r-1)+(n-1)(j^\prime-r)}\langle 1 ,f\circ t_n^{j^\prime-1}(a_0\cdc \mu_r(a_{n+i-j^\prime}\cdc a_{n+r-1+i-j^\prime})\cdc a_m) \rangle\\
&= \sum_{l=1}^{r}\sum_{i=1}^n (-1)^{lm+rm}(-1)^{(n-1)(i-1)}\langle 1 ,f(t_n^{i-1}(\mu_r(a_{l+n-1}\cdc a_m,a_0\cdc a_{l-1}),a_l\cdc a_{l+n-2})) \rangle \\
&\indent+\sum_{s=1}^{n-1}\sum_{j=1}^{s} (-1)^{sr+s}\langle 1 ,f\circ t_n^{j-1}(a_0\cdc\mu_r(a_s\cdc a_{s+r-1})\cdc a_m) \rangle\\
&\indent+\sum_{s^\prime=1}^{n-1}\sum_{j^\prime=s^\prime+1}^{n} (-1)^{(j^\prime-1)m+(n+j^\prime+s^\prime)(r-1)+(n-1)(j^\prime-r)}\langle 1 ,f\circ t_n^{j^\prime-1}(a_0\cdc \mu_r(a_{s^\prime}\cdc a_{s^\prime+r-1})\cdc a_m) \rangle\\
&=\sum_{l=1}^r\langle 1,(-1)^{lm+rm}(f\circ N_n)(\mu_r(a_{l+n-1}\cdc a_m,a_0\cdc a_{l-1}),a_l\cdc a_{l+n-1}) \rangle\\
&\indent+\sum_{s=1}^{n-1}\langle 1 ,(-1)^{sr+s}(f\circ N_n)(a_0\cdc a_{s-1},\mu_r(a_s\cdc a_{r+s-1}),a_{r+s}\cdc a_m) \rangle\\
&= \sum_{l=1}^r (-1)^{lm+rm}\langle \mu_r(a_{l+n-1}\cdc a_m,a_0\cdc a_{l-1}), \Delta(f)(a_l\cdc a_{l+n-1}) \rangle\\
&\indent+ \sum_{s=1}^{n-1}\langle a_0 ,(-1)^{sr+s}\Delta(f)(a_1\cdc a_{s-1},\mu_r(a_s\cdc a_{r+s-1}),a_{r+s}\cdc a_m) \rangle\\
&= \sum_{l=1}^r (-1)^{lm+rm}(-1)^{rdeg(f)}(-1)^{lr}\langle a_0,\mu_r(a_1,\cdc a_{l-1},\Delta(f)(a_l\cdc a_{l+n-1}),a_{l+n}\cdc a_m) \rangle\\
&\indent+\sum_{s=1}^{n-1}\langle a_0 ,(-1)^{sr+s}\Delta(f)\circ_s \mu_r(a_1\cdc a_m) \rangle\\
&= \sum_{l=1}^r (-1)^{rm+n+rdeg(f)}(-1)^{(l-1)n}\langle a_0, (\mu_r\circ_l\Delta(f))(a_1\cdc a_m) \rangle\\
&\indent+\sum_{s=1}^{n-1}\langle a_0 ,(-1)^{r+1}(-1)^{(s+1)(r+1)}\Delta(f)\circ_s \mu_r(a_1\cdc a_m) \rangle\\
&=(-1)^{rm+n+rdeg(f)+(r+1)deg(f)}\langle a_0, \mu_r\circ\Delta(f)(a_1\cdc a_m) \rangle-(-1)^{nr+r}\langle a_0 ,\Delta(f)\circ\mu_r(a_1\cdc a_m) \rangle\\
\end{align*}} So we have shown
\begin{equation*}
(-1)^{(n+1)r}\langle a_0, \Delta(f\circ\mu_r)(\alpha_m)\rangle=(-1)^{rm}(-1)^{||f||}\langle a_0, \mu_r\circ\Delta(f)(\alpha_m) \rangle-(-1)^{(n+1)r}\langle a_0 ,\Delta(f)\circ\mu_r(\alpha_m) \rangle
\end{equation*}
but, $(n+1)r+rm \equiv r^2+r\equiv 0$ thus
\begin{equation*}
\langle a_0, \Delta(f\circ\mu_r)(\alpha_m)\rangle-(-1)^{||f||}\langle a_0, \mu_r\circ\Delta(f)(\alpha_m)\rangle+\langle a_0 ,\Delta(f)\circ\mu_r(\alpha_m) \rangle=0
\end{equation*}
hence the claim.
\end{proof}
\begin{corollary}\label{bvdiffcor}  $\Delta d + d\Delta=0$.  
\end{corollary}

\subsection{Spined Braces}  In this section we will show the differential on the spined brace operations is compatible with the action $\rho$.  Since this calculation is lengthy, and since the two types of spined braces work similarly, we will consider a spined brace operations of type one.  First we must describe the trees which appear in $\partial_H(\beta_{l,n})$.

Define a collection of trees $\mu_m\circ_{i,v}\tau$ as follows.  This collection is empty unless both $l\geq i-1$ and $m-i\leq n-l$, in which case we define $\mu_m\circ_{i,v}\tau = \mu_m\circ_i \beta_{l-i+1,n-m+1}$.  Then we define
\begin{equation*}
\mu\circ_v\tau:=\sum_{m=2}^{n+1}\ds\sum_{i=1}^m\pm\sigma_i(\mu_m\circ_{i,v}\tau)
\end{equation*}
where we identify $\delta=\beta_{0,0}$ and where again $\sigma_i$ is the permutation $(1\dots i)$ which assures that the vertex $v$ is labeled by $1$ and the remaining vertices are labeled in the planar order.

In a similar fashion we define $\tau\circ_v\mu$ to be the sum over all $m\geq 2$ of the collections of trees which can be formed by breaking off $m$ consecutive non-root branches at $v$ which are not separated by the spine, gluing the root of $\mu_m$ on to $v$ in the region spanned by the broken branches and then identifying the broken branches with the white vertices of $\mu_m$ in the planar order.  Notice that in making identifications, no branch is moved past the spine.  Also notice that for $m> \text{max}\{l, n-l\}$ this set of trees is empty.

Finally we include those trees formed by pushing off the spine.  They are $\beta_n\wedge_l\mu_2$ and $\beta_n\wedge_{l+1}\mu_2$.  We can then describe $\partial_H(\beta_{l,n})$ as the following finite sum:
\begin{equation}\label{diffeq1}
\partial_H(\beta_{l,n}) = \mu\circ_v\beta_{l,n}\pm\beta_{l,n}\circ_v\mu\pm\beta_n\wedge_{l+1}\mu_2\mp\beta_n\wedge_l\mu_2
\end{equation}
See Figure $\ref{fig:boundary}$ for an example.  We can thus calculate:

\begin{equation*}
\rho(\partial_H(\beta_{l,n}))= d_{Int}(B_{l,n})
\end{equation*}
We will first calculate $\rho(\partial_H(\beta_{l,n}))$ via equation $\ref{diffeq1}$ in Lemmas $\ref{applem1}$, $\ref{applem2}$, and $\ref{applem3}$.  We will then calculate $d_{Int}(B_{l,n})$ and show the two coincide in Proposition $\ref{appprop}$.  This calculation will be done $mod$ $2$, with signs following similarly to Proposition $\ref{bvdiff}$. 

\begin{lemma} \label{applem1} Let $\alpha \in \mathbb{T}A$.  Then,
\begin{equation*}
\langle a_0,\rho(\beta_n\wedge_{l+1}\mu_2 - \beta_n\wedge_l\mu_2)(f,g_1\cdc g_n)(\alpha)\rangle=\langle 1,B^1_{l,n}(\mu_2\circ f;g_1 \cdc g_n)(a_0,\alpha) \rangle
\end{equation*}
\end{lemma}

\begin{proof} 
\begin{equation*}
\langle 1,B^1_{l,n}(\mu_2\circ_1 f;g_1\cdc g_n)(a_0,\alpha)\rangle=\langle 1,(\mu_2\circ_1 f)\{g_{l+1}\cdc g_n,a_0,g_1\cdc g_l\}(\alpha) \rangle
\end{equation*}
This expression represents a sum of terms which we can categorize by what is placed in the second leaf of $\mu_2$, where the two possibilities are a tensor factor of $\alpha$ or the function $g_l$.  Considering those terms which take $g_l$ on this leaf we have
\begin{align*}
&\langle 1,f\{g_{l+1}\cdc g_n,a_0,g_1\cdc g_{l-1}\}g_l(\alpha)\rangle=\langle 1,g_l[f\{g_{l+1}\cdc g_n,a_0,g_1\cdc g_{l-1}\}](\alpha) \rangle \\
&\indent=\langle a_0,\rho(\beta_n\wedge_l\mu_2)(f,g_1\cdc g_n)(\alpha) \rangle
\end{align*}
since all foliage will be killed by the pairing with $1$.  On the other hand if we consider $\langle 1,B^1_{l,n}(\mu_2\circ_2 f;g_1\cdc g_n)(a_0,\alpha)\rangle$ we can choose between putting a tensor factor of $\alpha$ or $g_{l+1}$ into the first leaf of $\mu_2$.  Choosing $g_{l+1}$ will give $\langle a_0,\rho(\beta_n\wedge_{l+1}\mu_2)(f,g_1\cdc g_n)(\alpha) \rangle$ and since those terms corresponding to placing a tensor factor into the open leaf of $\mu_2$ appear twice with the opposite sign, the extra terms will cancel, hence the lemma. 
\end{proof}
For the following argument we consider $B_{l,n}^1(f\circ\mu_m;g_1\cdc g_n)(a_0,-)$ as a sum of functions which we split in two
\begin{equation*}
B_{l,n}^1(f\circ\mu_m;g_1\cdc g_n)(a_0,-)=B_{l,n}^1(\overline{f}\circ\mu_m;g_1\cdc g_n)(a_0,-)+B_{l,n}^1(f\circ\overline{\mu_m};g_1\cdc g_n)(a_0,-)
\end{equation*}
where the first term is those functions where $a_0$ is evaluated in $f$ and the second term is those functions where $a_0$ is evaluated in $\mu_m$. 

\begin{lemma}\label{applem2}  Let $\alpha_m = \tensor_{i=1}^N a_i$ where $N=|f|-n+m-1+\sum_i |g_i|$.  Then,
\begin{align*}
\langle a_0,\rho(\mu\circ_v\beta_{l,n})(f;g_1\cdc g_n)(\alpha)\rangle
&=\bigoplus_{m\geq 1}\langle 1,B^1_{l,n}(f\circ\overline{\mu_m};g_1\cdc g_n)(a_0,\alpha_m) \rangle \\ 
&-\bigoplus_{m\geq 1}\ds\sum_{j=1}^m\langle 1,B^1_{l,n}(f;g_1\cdc g_n)(\mu_m(\dots, a_0, \dots),a_j,\dots) \rangle 
\end{align*}
\end{lemma}
\begin{proof}  First note that if $m=1$, no terms appear on the left hand side by definition, and the terms on the right hand side cancel.  Now suppose $m\geq2$.  On the top right hand side we can consider the terms of $B^1_{l,n}(f\circ\overline{\mu_m},g_1\cdc g_n)$ based on how many flags of $\mu_m$ are filled with functions $g_i$.  Those terms with no such leaves will cancel with the other term on the right hand side.  Those terms that are filled with functions (meaning $m-1$ leaves have a $g_i$ and the remaining leaf has $a_0$) come from the left hand side.  Finally consider those terms with $r$ of the $g_i$ glued to $\mu_m$ where $1\leq r \leq m-2$, but having $m-r-1$ free tails.  Such terms come from the left hand side in the form of $\mu_{r+1}\circ_v\beta_{l,n}$ after applying the foliage operator.
\end{proof}

\begin{lemma} \label{applem3}  Let $\alpha_m = \tensor_{i=1}^N a_i$ where $N=|f|-n+m-1+\sum_i |g_i|$.  Then,
\begin{align*}
&\langle a_0,\rho(\beta_{l,n}\circ_v\mu)(f;g_1\cdc g_n)(\alpha)\rangle\\
&\indent\indent=\langle 1,B^1_{l,n}(\overline{f}\circ\mu;g_1\cdc g_n)(a_0,\alpha) \rangle +\langle 1,\sum_{i=1}^nB^1_{l,n}(f;\dots,g_i\circ\mu-\mu\circ g_i\cdc g_n)(a_0,\alpha) \rangle\\
&\indent\indent-\bigoplus_{m\geq 1}\ds\sum_{i=1}^{N-m+1}\langle 1,[B^1_{l,n}(f;g_1\cdc g_n)] (a_0,a_1,\dots,a_{i-1},\mu_m(a_i,\dots,a_{i+m-1}),a_{i+m},\dots,a_N) \rangle
\end{align*}
\end{lemma}
\begin{proof}  This follows similarly to the previous lemma.  Write 
\begin{equation*}
B^1_{l,n}(\overline{f}\circ\mu_m;g_1\cdc g_n)_r(a_0,\alpha) \subset B^1_{l,n}(\overline{f}\circ\mu;g_1\cdc g_n)(a_0,\alpha)
\end{equation*}
where $r=1\cdc m$ for those terms having $r$ of the leaves of $\mu_m$ filled with functions $g_i$.  Then
\begin{align*}
&B^1_{l,n}(\overline{f}\circ\mu_m;g_1\cdc g_n)_0(a_0,\alpha) =  \\
& \indent \ds\sum_{i=1}^{N-m+1}[B^1_{l,n}(f;g_1\cdc g_n)](a_0,a_1,\dots,a_{i-1},\mu_m(a_i,\dots,a_{i+m-1}),a_{i+m}, \dots,a_N)\\ 
&\indent\indent-\sum_{i=1}^nB^1_{l,n}(f;\dots,g_i\circ\mu_m \cdc g_n)
\end{align*}
and
\begin{equation*}
\langle 1,B^1_{l,n}(\overline{f}\circ\mu_m;g_1\cdc g_n)_1(a_0,\alpha)\rangle =\langle 1,\sum_{i=1}^nB^1_{l,n}(f;g_1\cdc\mu_m\circ g_i\cdc g_n)(a_0,\alpha) \rangle \end{equation*}
and $\langle 1,B^1_{l,n}(\overline{f}\circ\mu_m,g_1\cdc g_n)_r(a_0,\alpha)\rangle$ for $r\geq 2$ come from the left hand side.  In particular if $m=r$ then the terms come directly and if $2\leq r <m$ then the term comes from applying the foliage operator to $\beta_{l,n}\circ_v\mu_r$.
\end{proof}

The preceding three lemmas allow us to prove the compatibility of the differentials with the spined braces of type $1$ as follows:
\begin{proposition}\label{appprop}  $\rho(\partial_H(\beta_{l,n}))=d_{Int}(\rho(\beta_{l,n}))$.
\end{proposition}

\begin{proof}  \begin{align*}
&\langle a_0,d_{Int}(\rho(\beta_{l,n}))(f;g_1\cdc g_n) (\alpha) \rangle \\
&=\langle a_0,[(\rho(\beta_{l,n})\circ\mu\pm\mu\circ \rho(\beta_{l,n}))(f,g_1\cdc g_n)\pm\sum_{i=0}^n \rho(\beta_{l,n})(f,\dots,g_i\circ\mu-\mu\circ g_i\cdc g_n)] (\alpha) \rangle \\ 
& = \langle a_0,(\rho(\beta_{l,n})\circ\mu\pm\mu\circ \rho(\beta_{l,n}))(f,g_1\cdc g_n)(\alpha)\rangle\pm\sum_{i=0}^n\langle 1,B^1_{l,n}(f;\dots,g_i\circ\mu-\mu\circ g_i\cdc g_n) (\alpha) \rangle\\
\end{align*}
using the notation $g_0=f$.  If we examine the sum in the last line we see that when $i=0$ the $\dots \mu\circ f \dots$ term can be reduced to considering $\dots \mu_2\circ f \dots$ which gives us the terms from $\ref{applem1}$ corresponding to pushing off the spine in the differential.  Also when $i=0$, the $\dots f\circ\mu\dots$ term splits in to two depending on the location of $a_0$.  These terms appear in lemmas $\ref{applem2}$ and $\ref{applem3}$.  For $i\geq1$ the terms appear in $\ref{applem3}$.  So it remains to show that the remaining terms on both sides are the same.  We consider then those terms which are as yet unaccounted for: 
\begin{align*}
\langle a_0&,\rho(\beta_{l,n})(f,g_1\cdc g_n)\circ\mu_m (a_1,\dots,a_N) \rangle-\langle a_0,\mu_m\circ \rho(\beta_{l,n})(f,g_1\cdc g_n) (a_1,\dots,a_N)\rangle \\ 
=& \ds\sum_{i=1}^{N-m+1}\langle a_0,[\rho(\beta_{l,n})(f,g_1\cdc g_n)\circ_i\mu_m] (a_1,\dots,a_N) \rangle \\
& \indent -\ds\sum_{j=1}^m\langle a_0,[\mu_m\circ_j \rho(\beta_{l,n})(f,g_1\cdc g_n)] (a_1,\dots,a_N) \rangle \\
=& \ds\sum_{i=1}^{N-m+1}\langle a_0,[\rho(\beta_{l,n})(f,g_1\cdc g_n)] (a_1,\dots,a_{i-1},\mu_m(a_i,\dots,a_{i+m-1}),a_{i+m},\dots,a_N) \rangle\\
& \indent -\ds\sum_{j=1}^m\langle a_0,\mu_m(a_1,\dots, a_{j-1}, \rho(\beta_{l,n})(f_1,\dots, f_n)(a_j,\dots,a_{N-j+m}),a_{N-j+m+1},\dots,a_N) \rangle \\
=& \ds\sum_{i=1}^{N-m+1}\langle 1,[B^1_{l,n}(f;g_1\cdc g_n)] (a_0,a_1,\dots,a_{i-1},\mu_m(a_i,\dots,a_{i+m-1}),a_{i+m},\dots,a_N) \rangle\\
& \indent \pm\ds\sum_{j=1}^m\langle \rho(\beta_{l,n})(f,g_1\cdc g_n)(a_j,\dots,a_{N-j+m}),\mu_m(a_{N-j+m+1},\dots, a_N, a_0, \dots, a_{j-1}) \rangle \\
=& \ds\sum_{i=1}^{N-m+1}\langle 1,[B^1_{l,n}(f;g_1\cdc g_n)] (a_0,a_1,\dots,a_{i-1},\mu_m(a_i,\dots,a_{i+m-1}),a_{i+m},\dots,a_N) \rangle\\
& \indent \pm\ds\sum_{j=1}^m\langle 1,B^1_{l,n}(f;g_1\cdc g_n)(\mu_m(a_{N-j+m+1},\dots, a_N, a_0, \dots, a_{j-1}),a_j,\dots,a_{N-j+m}) \rangle 
\end{align*}
Notice that these sums appear in Lemmas $\ref{applem3}$ and $\ref{applem2}$ respectively.  Now, the terms appearing in Lemmas $\ref{applem1}$, $\ref{applem2}$, and $\ref{applem3}$ amount to $\rho(\partial_H(\beta_{l,n}))$ on the one hand, and by the above calculation they amount to $d_{Int}(\rho(\beta_{l,n}))$, hence the proposition.
\end{proof}

\bibliography{cycbib}
\bibliographystyle{amsalpha}

\end{document}